\documentclass[11pt]{amsart}
\usepackage[letterpaper,margin=1.2in]{geometry}
\usepackage{amsbsy,amsfonts,amsmath,amssymb,amscd,amsthm,mathrsfs}
\usepackage{subfigure,enumitem,graphicx,epsfig,color}
\usepackage[colorlinks=true,linkcolor=blue,citecolor=green]{hyperref}
\normalsize

\newtheorem{theorem}{Theorem}[section]
\newtheorem{lemma}[theorem]{Lemma}

\newtheorem{definition}[theorem]{Definition}
\newtheorem{corollary}[theorem]{Corollary}
\newtheorem{remark}[theorem]{Remark}

\newtheorem{proposition}[theorem]{Proposition}

\newtheorem*{theorem A}{Theorem A}
\newtheorem*{corollary B}{Corollary B}
\newtheorem*{corollary C}{Corollary C}
\newtheorem*{theorem B}{Theorem B}
\theoremstyle{definition}
\newtheorem{example}[theorem]{Example}

\begin{document}
\title[Specification and thermodynamic properties of topological NDSs]{Specification and thermodynamic properties of topological time-dependent dynamical systems}
\author[J. Nazarian Sarkooh]{J. Nazarian Sarkooh}
\author[F. H. Ghane]{F. H. Ghane$^{*}$}
\address{Department of Mathematics, Ferdowsi University of Mashhad, Mashhad, IRAN.}
\email{\textcolor[rgb]{0.00,0.00,0.84}{javad.nazariansarkooh@gmail.com}}
\address{Department of Mathematics, Ferdowsi University of Mashhad, Mashhad, IRAN.}
\email{\textcolor[rgb]{0.00,0.00,0.84}{ghane@math.um.ac.ir}}
 \thanks{$^*$Corresponding author}
\begin{abstract}
This paper discusses the thermodynamic properties for certain time-dependent dynamical
systems. In particular, we are interested in time-dependent dynamical systems with the specification property.
We show that each time-dependent dynamical system given by a sequence of surjective continuous self maps of a compact metric space with the specification property has positive topological entropy and all points are entropy point. In particular, it is proved that these systems are topologically chaotic. We will treat the dynamics of uniformly Ruelle-expanding time-dependent dynamical systems on compact metric spaces and provide some sufficient conditions that these systems have the specification property. Consequently, we conclude that these systems have positive topological entropy. This extends a result of Kawan \cite{K2}, corresponding to the case when the expanding maps are smooth, to the more general case of expanding maps. Additionally, we study the topological pressure of time-dependent dynamical systems. We obtain conditions under which the topological entropy and topological pressure of any continuous potential can be computed as a limit at a definite size scale. Finally, we study the Lipschitz regularity of the topological pressure function for expansive (uniformly Ruelle-expanding) time-dependent dynamical systems on compact metric spaces.
\end{abstract}

\maketitle
\thispagestyle{empty}

\section{Introduction}
The time-dependent systems so-called non-autonomous, yield very flexible models than
autonomous cases for the study and description of real world processes.
They may be used
to describe the evolution of a wider class of phenomena, including systems which are forced
or driven. Recently, there have been major efforts in establishing a general theory of such
systems (see \cite{BV,K1,K2,K3,KR,KS,O,DTRD}), but a global theory is still out of reach. Our main goal in
this paper is to describe the topological aspects of the thermodynamic formalism for time-dependent dynamical systems.

Thermodynamic formalism, that is the formalism of equilibrium statistical physics, was adapted to the theory of dynamical systems in the classical works of Sinai, Ruelle and Bowen in the 1970s \cite{B3,DR1,DRYS,YS}.
Topological pressure, topological entropy, Gibbs Measures and equilibrium states are the fundamental notions in thermodynamic formalism. Topological pressure is the main tool in studying dimension of invariant sets and measures for dynamical systems in dimension theory. The notion of entropy, on the other hand, is one of the most important objects in dynamical systems, either as a topological invariant or as a measure of the chaoticity of dynamical systems.
For uniformly hyperbolic systems Bowen \cite{B3} presented a complete description of thermodynamic formalism, but in the non-uniformly hyperbolic case a general theory of thermodynamic formalism, despite substantial progress by several authors, is far from being complete.

In the theory of dynamical systems, topological entropy is a nonnegative extended real number measuring the complexity of a topological dynamical system. For the autonomous case, topological entropy was first introduced by Adler, Konhelm and McAndrew, via open covers for continuous maps in compact topological spaces \cite{AKM}.
In 1970, Bowen gave another definition in terms of separated and spanning sets for uniformly continuous maps in metric spaces \cite{RB}, and this definition is equivalent to Adler’s definition for continuous maps in compact metric spaces. Topological entropy has
close relationships with many important dynamical properties, such as chaos, Lyapunove exponents, the growth of the number of periodic points and so on. Moreover, positive topological entropy have remarkable role in the characterization of the dynamical behaviors; for instance, Downarowicz proved that positive topological entropy implies chaos DC2 \cite{TD}. Thus, a lot of attention has been focused on computations and
estimations of topological entropy of an autonomous dynamical system and many good results have been
obtained \cite{LBWC,BGM,RB1,RB,LWG,TNTG,SI,AM,MMFP}.
In 1996, Kolyada and Snoha extended the concept of topological entropy to time-dependent systems, based on open covers, separated sets and spanning sets, and obtained a series of important properties of these systems \cite{KS}. Recently, Kawan \cite{K1,K2,K3} introduced and studied the notion of metric entropy for
non-autonomous dynamical systems and showed that it is related via variational inequality (principle) to the topological entropy as defined by Kolyada and Snoha. More precisely, for equicontinuous topological
non-autonomous dynamical systems
$(X_{1,\infty},f_{1,\infty})$, he proved the variational inequality
\begin{equation*}
\sup_{\mu_{1,\infty}}h_{\mathcal{E}_{M}}(f_{1,\infty},\mu_{1,\infty})\leq h_{\text{top}}(f_{1,\infty}),
\end{equation*}
where $h_{\text{top}}(f_{1,\infty})$ and $h_{\mathcal{E}_{M}}(f_{1,\infty},\mu_{1,\infty})$ are the topological
and metric entropy of $(X_{1,\infty},f_{1,\infty})$, the supremum is taken over all invariant measure sequences $\mu_{1,\infty}$ and $\mathcal{E}_{M}$ is Misiurewicz class of partitions. Also, for NDSs $(M,f_{1,\infty})$ built from $C^{1}$ expanding maps $f_{n}$ on a compact Riemannian manifold $M$ with uniform bounds on expansion factors and derivatives that act in the same way on the fundamental group of $M$, he proved the full variational principle
\begin{equation*}
\sup_{\mu_{1,\infty}}h_{\mathcal{E}_{M}}(f_{1,\infty},\mu_{1,\infty})=h_{\text{top}}(f_{1,\infty}).
\end{equation*}

On the other hand, some authors provided conditions such that the dynamic has
positive topological entropy. Shao $et\ al.$ \cite{HWZ} have given an estimation of lower bound of topological
entropy for coupled-expanding systems associated with transition matrices in compact Hausdorff spaces.
Rodrigues and Varandas showed that any finitely generated continuous semigroup
action on a compact metric space with the strong orbital specification property has positive topological entropy \cite{RV}. Also they proved that if
each element of the semigroup action is local homeomorphism and semigroup enjoying the strong orbital
specification property, then every point is an entropy point. Entropy points are those that their
local neighborhoods reflect the complexity of the entire dynamical system from the viewpoint of
entropy theory.

The objective of this paper is to extend these results to time-dependent dynamical systems. We show that
any time-dependent dynamical system of surjective maps with the specification property
has positive topological entropy and all points are entropy point. In particular, these dynamical systems are topologically chaotic. Additionally, we discuss the thermodynamic properties of uniformly Ruelle-expanding 
time-dependent dynamical systems on compact metric spaces. We provide some sufficient conditions that these 
kind of time dependent dynamical systems having the specification property; consequently, these systems have positive topological entropy. This extends a result of Kawan \cite{K2} which states that each smooth expanding time-dependent system has positive topological entropy.

For the autonomous case, the definition of topological pressure by using the separated sets was introduced by Ruelle \cite{DR}. Later,  Walters \cite{PW1} presented equivalent approach to the notion of topological pressure via open covers and spanning sets. In 2008, Huang $et\ al.$ \cite{HWZ} extended these definitions to non-autonomous
dynamical systems, by using open covers, separated and spanning sets. Recently, Kawan introduced the
notion of metric pressure for a non-autonomous dynamical system which is related via a variational inequality to
the topological pressure of non-autonomous dynamical systems \cite{K2}.

The topological pressure can be computed as the limiting complexity of the dynamical system as the size scale
approaches zero. Nevertheless, some authors imposed some conditions so that the topological pressure of the system can be computed as a limit at a definite size scale. For instance, in the context of semigroup actions, Rodrigues and Varandas showed that the topological pressure of any continuous potential that satisfies the uniformly bounded variation condition can be computed as a limit at a definite size scale for any finitely generated continuous semigroup action on a compact metric space having the
$\ast$-expansive property \cite{RV}. With the same motivation, we extend these results to time-dependent dynamical systems. More precisely, we show that the topological entropy and topological pressure of any continuous potential can be computed as a limit at a definite size scale whenever the time-dependent system satisfies the $\ast$-expansive property. Moreover, we prove a strong regularity of the topological pressure function.

\textbf{This is how the paper is organized:} In Section \ref{section2}, we give a precise definition of a
time-dependent dynamical system, review the main concepts and set up our notation.
In Section \ref{section3}, we characterize entropy points of time-dependent dynamical systems with the specification property and show that any time-dependent system given by a sequence of surjective continuous self maps of a compact metric space with the specification property has positive topological entropy and all points are entropy point; in particular, this system is topologically chaotic; this means that it has positive asymptotical topological entropy. Uniformly Ruelle-expanding time-dependent dynamical systems are discussed in Section \ref{section4}. By adding some conditions, we show that uniformly Ruelle-expanding time-dependent systems on compact metric spaces possess the specification property. In Section \ref{section5}, we study the topological pressure of time-dependent systems. We introduce a special class of continuous potentials and provide another formula to compute the topological pressure of this class of potentials. Then, we obtain conditions under which the topological entropy and topological pressure of any continuous potential can be computed as a limit at a definite size scale. Additionally, we study the Lipschitz regularity of the topological pressure function for expansive (uniformly Ruelle-expanding) time-dependent dynamical systems on compact metric spaces. Finally, in Section \ref{applications}, we provide several examples of time-dependent systems which fit in our situation.
\section{Preliminaries}\label{section2}
A \emph{time-dependent} or \emph{non-autonomous} dynamical system (an \emph{NDS} for short), is a pair $(X_{1,\infty}, f_{1,\infty})$, where $X_{1,\infty}=(X_n)_{n=1}^\infty$ is a sequence of sets
and $f_{1,\infty}=(f_n)_{n=1}^\infty$ is a sequence of maps $f_n : X_n \to X_{n+1}$.
By $(X_{n,\infty},f_{n,\infty})$, we denote the pair of shifted sequences
 $X_{n,\infty}=(X_{n+k})_{k=0}^{\infty}$, $f_{n,\infty}=(f_{n+k})_{k=0}^{\infty}$
and we use similar notation for other sequences associated with an NDS.
If all the sets $X_n$ are compact metric spaces and all the $f_n$ are continuous,
we say that $(X_{1,\infty}, f_{1,\infty})$ is a \emph{topological NDS}.
Here, we assume that $X$ is a compact metric space with metric $d$, all the sets $X_{n}$ are equal to the set $X$ and we abbreviate $(X_{1,\infty}, f_{1,\infty})$ by $(X, f_{1,\infty})$. Throughout this paper we work with topological NDSs and use NDS instead of topological NDS for simplicity.
The time evolution of the system is defined by composing the maps $f_{n}$ in the obvious way.
In general, we define
\begin{equation*}
f_k^n:=f_{k+n-1}\circ\cdots\circ f_{k+1}\circ f_k \ \text{for} \ k,n\in \mathbb{N}, \ \text{and} \ f_k^0:=id_X.
\end{equation*}
We also put $f_k^{-n}:=(f_k^n)^{-1}$, which is only applied to subsets $A \subset X$.
The \emph{trajectory} of a point $x \in X$ is the sequence $(f_1^n(x))_{n=0}^\infty$.

For an NDS $(X, f_{1,\infty})$ on a compact metric space $(X,d)$, the \emph{Bowen-metrics} on $X$ are given by
\begin{equation}\label{eq8}
d_{k,n}(x,y):=\max_{0\leq i \leq n}d(f_{k}^{i}(x),f_{k}^{i}(y))\ \text{for}\ k\geq 1\ \text{and}\ n\geq 0.
\end{equation}
Also, for any $k\geq 1$, $n\geq 0$, $x\in X$ and $\epsilon >0$, we define
\begin{equation}\label{01}
B(x,k,n,\epsilon):=\{y \in X: d_{k,n}(x,y)<\epsilon\},
\end{equation}
which is called a \emph{dynamical} $(n+1)$-\emph{ball} with initial time $k$.

Fix an NDS $(X, f_{1,\infty})$  and $k\geq 1$. Based on \cite{KS}, for the NDS $(X, f_{k,\infty})$, the \emph{topological entropy} $h_{\text{top}}(X, f_{k,\infty})$ is defined as follows. A family $\mathcal{A}$ of subsets of $X$ is called a \emph{cover} (of $X$) if their union is all of the space $X$. For open covers $\mathcal{A}_{1},\mathcal{A}_{2},\ldots,\mathcal{A}_{n}$ of $X$ we denote
$$\bigvee_{i=1}^{n}\mathcal{A}_{i}=\mathcal{A}_{1}\vee\mathcal{A}_{2}\vee\cdots\vee\mathcal{A}_{n}=\big\{A_{1}\cap A_{2}\cap\cdots\cap A_{n}:A_{i}\in\mathcal{A}_{i}\ \text{for}\ i=1,\ldots,n\big\}.$$

Note that $\bigvee_{i=1}^{n}\mathcal{A}_{i}$ is also an open cover. For an open cover $\mathcal{A}$ we denote $f_{i}^{-n}(\mathcal{A})=\{f_{i}^{-n}(A):A\in\mathcal{A}\}$
and $\mathcal{A}_{i}^{n}=\bigvee_{j=0}^{n-1}f_{i}^{-j}(\mathcal{A})$. For each $j$, $f_{i}^{-j}(\mathcal{A})$ is an open cover, so $\mathcal{A}_{i}^{n}$ is also an open cover. Next, we denote by $\mathcal{N}(\mathcal{A})$ the \emph{smallest possible cardinality of a subcover chosen from} $\mathcal{A}$.

Then
$$h(X,f_{k,\infty},\mathcal{A}):=\limsup_{n\to\infty}\dfrac{1}{n}\log\mathcal{N}(\mathcal{A}_{k}^{n})$$
is said to be the \emph{topological entropy} of the NDS $(X, f_{k,\infty})$ on the cover $\mathcal{A}$.
The \emph{topological entropy} of the NDS $(X, f_{k,\infty})$ is defined by
$$h_{\text{top}}(X, f_{k,\infty}):=\sup\{h(X, f_{k,\infty},\mathcal{A}): \mathcal{A}\ \text{is a open cover of}\ X\}.$$

For introducing entropy points we need the following extension of the definition of topological entropy.
Let $Y$ be a nonempty subset of $X$. The set $Y$ may not be compact or may not exhibit any kind of invariance with respect to $f_{k,\infty}$. If $\mathcal{A}$ is a cover of $X$ we denote by $\mathcal{A}|_{Y}$ the cover $\{A\cap Y:A\in\mathcal{A}\}$ of the set $Y$. Then we define the \emph{topological entropy} of
the NDS $(X, f_{k,\infty})$ on the set $Y$ by
$$h_{\text{top}}(Y, f_{k,\infty}):=\sup\{h(Y, f_{k,\infty},\mathcal{A}): \mathcal{A}\ \text{is a open cover of}\ X\}$$
where
$$h(Y,f_{k,\infty},\mathcal{A}):=\limsup_{n\to\infty}\dfrac{1}{n}\log\mathcal{N}(\mathcal{A}_{k}^{n}|_{Y}).$$

Now, we consider the other definition for topological entropy of the NDS $(X, f_{k,\infty})$ by using separated and spanning sets.

A subset $E$ of the space $X$ is called $(n,\epsilon;f_{k,\infty})$\emph{-separated} if for any two
distinct points $x,y\in E$, $d_{k,n}(x,y)>\epsilon$. A set $F\subset X\ (n,\epsilon;f_{k,\infty})$\emph{-spans} another set $K\subset X$ provided that for each $x\in K$ there is $y\in F$ for
which $d_{k,n}(x,y)\leq\epsilon$.
For a subset $Y$ of $X$ we define $s_{n}(Y;f_{k,\infty};\epsilon)$ as the maximum cardinality
of an $(n,\epsilon;f_{k,\infty})$\emph{-separated set} in $Y$ and $r_{n}(Y;f_{k,\infty};\epsilon)$ as the smallest cardinality of a set in $Y$ which $(n,\epsilon;f_{k,\infty})$\emph{-spans} $Y$. If $Y=X$ we sometime suppress $Y$ and shortly write $s_{n}(f_{k,\infty};\epsilon)$ and $r_{n}(f_{k,\infty};\epsilon)$.
Following \cite{KS} as in the autonomous case, it can be shown that
\begin{equation}\label{eq4}
h_{\text{top}}(Y, f_{k,\infty})=\lim_{\epsilon\to 0}\limsup_{n\to\infty}\dfrac{1}{n}\log s_{n}(Y;f_{k,\infty};\epsilon)=\lim_{\epsilon\to 0}\limsup_{n\to\infty}\dfrac{1}{n}\log r_{n}(Y;f_{k,\infty};\epsilon).
\end{equation}

An autonomous dynamical system $(X,f)$ is called \emph{topologically chaotic} if $h_{\text{top}}(f)>0$.
But for the non-autonomous case, we use the definition that given by Kolyada and Snoha in \cite{KS}.

Let an NDS $(X,f_{1,\infty})$ and open cover $\mathcal{A}$ of $X$ be given, then
by \cite[Lemma 4.5]{KS} the limit
$$h^{*}(X,f_{\infty},\mathcal{A}):=\lim_{n\to\infty}h(X,f_{n,\infty},\mathcal{A})$$
exists. The quantity $h^{*}(X,f_{\infty},\mathcal{A})$ is said to be the \emph{asymptotical topological
entropy of the sequence} $f_{1,\infty}$
\emph{on the cover} $\mathcal{A}$. Put
$$h^{*}(X,f_{\infty}):=\sup_{\mathcal{A}}h^{*}(X,f_{\infty},\mathcal{A})$$
where the supremum is taken over all open covers $\mathcal{A}$ of $X$. In \cite{KS}, Kolyada and Snoha showed
that
\begin{eqnarray*}
h^{*}(X,f_{\infty})
=\lim_{n\to\infty}h_{\text{top}}(X, f_{n,\infty})
&=&\lim_{\epsilon\to 0}\lim_{n\to\infty}\limsup_{k\to\infty}\frac{1}{k}\log s_{k}(f_{n,\infty};\epsilon)\\
&=& \lim_{\epsilon\to 0}\lim_{n\to\infty}\limsup_{k\to\infty}\frac{1}{k}\log r_{k}(f_{n,\infty};\epsilon).
\end{eqnarray*}

The quantity $h^{*}(X,f_{\infty})$ will be said to be the \emph{asymptotical topological entropy} of the
sequence $f_{1,\infty}$. (In the notations $h^{*}(X,f_{\infty},\mathcal{A})$ and
$h^{*}(X,f_{\infty})$ we use the symbol $f_{\infty}$ instead of more precise $f_{1,\infty}$, because these quantities do not depend on whether we consider $f_{1,\infty}$ or $f_{i,\infty}$
for some $i\in\mathbb{N}$.)
\begin{definition}[Topologically chaotic]
An NDS $(X,f_{1,\infty})$ is said to be \emph{topologically chaotic} if it has positive asymptotical topological
entropy, i.e. $h^{*}(X,f_{\infty})>0$.
\end{definition}
\section{Specification property and topological entropy}\label{section3}
The notion of entropy is one of the most important objects in dynamical systems, either as a topological invariant or as a measure of complexity of the dynamics. Several notions of entropy had been introduced for dynamical systems in an attempt to describe its dynamical characteristics. In this section, we characterize entropy points of NDSs with the specification property and show that any NDS of surjective maps enjoying the specification property has positive topological entropy and all points are entropy point (Theorems \ref{theorem3} and \ref{theorem4}). In particular, it is topologically chaotic (Corollary \ref{corollary1}).
\subsection{Specification property and entropy points}
The notion of entropy point for finitely generated pseudogroup actions and finitely generated group actions was introduced, respectively, by Bi\'s \cite{AB} and Rodrigues and Varandas \cite{RV}. A finitely generated pseudogroup actions (group actions) $(G,G_{1})$ with generator set $G_{1}$ acting on a compact metric space $X$ admits an \emph{entropy point} $x_{0}$ if for every open neighbourhood $U$ of $x_{0}$ the equality
$h_{\text{top}}((G,G_{1}),\text{cl}(U))=h_{\text{top}}((G,G_{1}),X)$ holds. Note that, entropy points are those for which local neighborhoods reflect the complexity of the entire dynamical system in the context of topological entropy.
Now, we extend the notion of entropy point to NDSs.
\begin{definition}[Entropy point]
An NDS $(X, f_{1,\infty})$ admits an \emph{entropy point} $x_{0}$ if for any open neighbourhood $U$ of $x_{0}$ the equality $h_{\text{top}}(\text{cl}(U), f_{1,\infty})=h_{\text{top}}(X, f_{1,\infty})$ holds.
\end{definition}
Bi\'s in \cite[Theorem 2.5 and Corollary 2.6]{AB} proved remarkably that any finitely generated pseudogroup $(G,G_{1})$ acting on a compact metric space admits an entropy point. As a direct consequence of the proof of \cite[Theorem 2.5]{AB}, we have the following proposition (note that the proof of \cite[Theorem 2.5]{AB} does not require invertibility).
\begin{proposition}\label{theorem2}
Every NDS $(X, f_{1,\infty})$ admits an entropy point.
\end{proposition}
The notion of specification property for a continuous map on a compact metric space X was introduced
by Bowen \cite{RB3}. A continuous self-map $f$ acting on a compact metric space $(X,d)$ satisfies the
\emph{specification property} if for any $\epsilon>0$ there exists a positive integer $N=N(\epsilon)$ such that the following holds: for any integer $s\geq 2$, any $s$ points $x_{1},x_{2},\ldots,x_{s}\in X$ and any sequence of integers $1=j_{1}\leq k_{1}<j_{2}\leq k_{2}<\cdots<j_{s}\leq k_{s}$ with $j_{n}-k_{n-1}\geq N$
for $n=2,\ldots,s$, there is a point $x\in X$ such that
\begin{center}
$d\big(f^{j_{m}+i-1}(x),f^{i}(x_{m})\big)\leq\epsilon\ \text{for all}\ 0\leq i\leq k_{m}-j_{m}\ \text{and}\
1\leq m\leq s$.
\end{center}

Specification property means that pieces of orbits of $f$ can be $\epsilon$-shadowed by an individual orbit provided that the time lag between each shadowing is any prefixed time larger than $N(\epsilon)$. In the following definition we give a notion of specification property to NDSs that extends the above definition.
\begin{definition}[Specification property]\label{SP}
An NDS $(X, f_{1,\infty})$ has the \emph{specification property} if for any $\epsilon>0$ there is a positive integer $N=N(\epsilon)$ such that the following holds: for any integer $s\geq 2$, any $s$ points $x_{1},x_{2},\ldots,x_{s}\in X$ and any sequence of integers $1=j_{1}\leq k_{1}<j_{2}\leq k_{2}<\cdots<j_{s}\leq k_{s}$ with
$j_{n}-k_{n-1}\geq N$ for $n=2,\ldots,s$, there is a point $x\in X$ such that
\begin{equation}\label{eq11}
d\big(f_{1}^{j_{m}+i-1}(x),f_{j_{m}}^{i}(x_{m})\big)\leq\epsilon\ \text{for all}\ 0\leq i\leq k_{m}-j_{m}\ \text{and}\ 1\leq m\leq s.
\end{equation}
\end{definition}
\begin{remark}
Note that, if $(X, f_{1,\infty})$ is an NDS of surjective maps on a compact metric space $(X,d)$, then condition (\ref{eq11}) can be replaced with
\begin{equation}\label{eq12}
d\big(f_{1}^{i}(x),f_{1}^{i}(x_{m})\big)\leq\epsilon\ \text{for all}\ j_{m}-1\leq i\leq k_{m}-1\ \text{and}\ 1\leq m\leq s,
\end{equation}
because we can substitute the point $x_{m}$ with some point of $\{f_{1}^{1-j_{m}}(x_{m})\}$ for every $1\leq m\leq s$.
Hence, for simplicity, in the context of NDSs of surjective maps, we use the condition (\ref{eq12}) instead of (\ref{eq11}).
\end{remark}
In what follows, we show that the specification property for NDSs  is a sufficient condition for that all points are entropy point. Moreover, it yields the stronger that local complexity coincides with $h^{*}(X, f_{\infty})$. More precisely, we have the following theorem.
\begin{theorem}\label{theorem3}
Let $(X, f_{1,\infty})$ be an NDS of surjective maps on a compact metric space $X$ without any isolated point. If the NDS $(X, f_{1,\infty})$ satisfies the specification property, then any point of $X$ is an entropy point and $h_{\text{top}}(X, f_{1,\infty})=h_{\text{top}}(X, f_{k,\infty})$ for every $k\geq 1$.
In particular, $h_{\text{top}}(X, f_{1,\infty})=h^{*}(X, f_{\infty})$.
\end{theorem}
\begin{proof}
By \cite[Lemma 4.5]{KS} we have $h_{\text{top}}(X, f_{1,\infty})\leq h_{\text{top}}(X, f_{k,\infty})$
for all $k\geq 1$. Also, by relation (\ref{eq4}) we know that
$$h_{\text{top}}(X, f_{k,\infty})=\lim_{\epsilon\to 0}\limsup_{n\to\infty}\dfrac{1}{n}\log s_{n}(f_{k,\infty};\epsilon).$$

For any $z\in X$ and any open neighborhood $V$ of $z$ we show that
$h_{\text{top}}(cl(V), f_{1,\infty})=h_{\text{top}}(X, f_{1,\infty})$.
Let $k\geq 1$ and define $W_{\epsilon}:=\{y\in V: d(y,\partial V)>\dfrac{\epsilon}{4}\}$ for $\epsilon>0$.
Now, fix $\epsilon>0$ such that the open set $W_{\epsilon}$ be nonempty. Assume that
\begin{itemize}
\item
$N(\frac{\epsilon}{4})\geq 1$ is given by the specification property;
\item
 $E:=\{w_{1}, w_{2},\ldots, w_{l}\}\subseteq X$ is a
maximal $(n,\epsilon;f_{N(\frac{\epsilon}{4})+(k+1),\infty})$-separated set;
\item
$E^{\prime}=\{w_{1}^{\prime}, w_{2}^{\prime},\ldots, w_{l}^{\prime}\}\subseteq X$ is a preimage
set of $E$ under $f_{1}^{N(\frac{\epsilon}{4})+k}$, i.e. $f_{1}^{N(\frac{\epsilon}{4})+k}(w_{i}^{\prime})=w_{i}$ for $1\leq i\leq l$;
\item
$y\in W_{\epsilon}$ is an arbitrary point ($W_{\epsilon}\neq \emptyset$, because $X$ does not have any isolated point).
\end{itemize}

Let $j_{1}=k_{1}=1$, $j_{2}=N(\frac{\epsilon}{4})+k+1$ and $k_{2}=N(\frac{\epsilon}{4})+k+n+1$. By the definition of specification property for each $w_{i}^{\prime}\in E^{\prime}$ by taking $x_{1}=y$ and $x_{2}=w_{i}^{\prime}$, there exists $y_{i}\in B(y,\frac{\epsilon}{4})$ such that
$f_{1}^{N(\frac{\epsilon}{4})+k}(y_{i})\in B(f_{1}^{N(\frac{\epsilon}{4})+k}(w_{i}^{\prime}),N(\frac{\epsilon}{4})+(k+1),n,\frac{\epsilon}{4})=B(w_{i},N(\frac{\epsilon}{4})+(k+1),n,\frac{\epsilon}{4})$.
Since $E:=\{w_{1},w_{2},\ldots,w_{l}\}\subseteq X$ is a
maximal $(n,\epsilon;f_{N(\frac{\epsilon}{4})+(k+1),\infty})$-separated set,
then the set $\{y_{i}\}_{i=1}^{l}\subseteq cl(V)$ is
an $(N(\frac{\epsilon}{4})+k+n,\frac{\epsilon}{2};f_{1,\infty})$-separated set. So
\begin{center}
$s_{N(\frac{\epsilon}{4})+k+n}(cl(V);f_{1,\infty};\frac{\epsilon}{2})\geq
s_{n}(f_{N(\frac{\epsilon}{4})+(k+1),\infty};\epsilon)$.
\end{center}
Thus, we have
\begin{eqnarray*}
h_{\text{top}}(cl(V),f_{1,\infty})
&\geq & \limsup_{n\to\infty}\dfrac{1}{N(\frac{\epsilon}{4})+k+n}\log s_{N(\frac{\epsilon}{4})+k+n}
(cl(V);f_{1,\infty};\frac{\epsilon}{2})\\
&\geq &\limsup_{n\to\infty}\dfrac{1}{N(\frac{\epsilon}{4})+k+n}\log s_{n}(f_{N(\frac{\epsilon}{4})+(k+1),\infty};\epsilon)\\
&= & \limsup_{n\to\infty}\dfrac{1}{n}\log s_{n}(f_{N(\frac{\epsilon}{4})+(k+1),\infty};\epsilon).
\end{eqnarray*}
This implies that
\begin{eqnarray}\label{05}
h_{\text{top}}(X,f_{1,\infty})\geq h_{\text{top}}(cl(V),f_{1,\infty})\geq h_{\text{top}}(X,f_{N(\frac{\epsilon}{4})+(k+1),\infty})\geq h_{\text{top}}(X,f_{1,\infty}).
\end{eqnarray}
Hence we have $h_{\text{top}}(cl(V), f_{1,\infty})=h_{\text{top}}(X, f_{1,\infty})$, i.e. $z$ is an entropy point.
Since $z\in X$ was arbitrary, it follows that every point of $X$ is an entropy point.

On the other hand, since $k\geq 1$ was arbitrary and $h_{\text{top}}(X, f_{1,\infty})\leq h_{\text{top}}(X, f_{k,\infty})$, the relation (\ref{05}) implies that $h_{\text{top}}(X, f_{1,\infty})=h_{\text{top}}(X, f_{k,\infty})$ for every $k\geq 1$. Consequently, $h_{\text{top}}(X, f_{1,\infty})=h^{*}(X, f_{\infty})$ which completes the proof of the theorem.
\end{proof}
As an application of Theorem \ref{theorem3}, the following corollary is a criterion under which an NDS does not have the specification property.
\begin{corollary}\label{corollary2}
Let $(X, f_{1,\infty})$ be an NDS of surjective maps on a compact metric space $X$ without any isolated point. If $h^{*}(X, f_{\infty})>h_{\text{top}}(X, f_{1,\infty})$ or $h_{\text{top}}(X, f_{i,\infty})>h_{\text{top}}(X, f_{j,\infty})$ for some $i>j\geq 1$, then NDS $(X, f_{1,\infty})$ does not have the specification property.
\end{corollary}
\subsection{Specification property and positive topological entropy}
In this subsection we show that any NDS of surjective maps with the specification property has positive topological entropy. In the other words, the specification property is enough to guarantee that any NDS of surjective maps has positive topological entropy. Moreover, we prove that any NDS of surjective maps with the specification property is topologically chaotic.
\begin{theorem}\label{theorem4}
Let $(X, f_{1,\infty})$ be an NDS of surjective maps on a compact metric spase $X$ without any isolated point. If the NDS  $(X, f_{1,\infty})$ satisfies the specification property, then it has positive topological entropy, i.e. $h_{\text{top}}(X,f_{1,\infty})>0$.
\end{theorem}
\begin{proof}
By relation (\ref{eq4}) we know that
$$h_{\text{top}}(X,f_{1,\infty})=\lim_{\epsilon\to 0}\limsup_{_{n\to\infty}}\dfrac{1}{n}
\log s_{n}(f_{1,\infty};\epsilon)$$
where the limit can be replaced by $\sup_{\epsilon>0}$. Hence, it is enough to prove that there exists $\epsilon>0$ small so that
$$\limsup_{_{n\to\infty}}\dfrac{1}{n}\log s_{n}(f_{1,\infty};\epsilon)>0.$$

Let $\epsilon>0$ be small and fixed so that there are at least two distinct $2\epsilon$-separated
points $x_{1},y_{1}\in X$, i.e. $d(x_{1},y_{1})>2\epsilon$ (note that $X$ has no any isolated point). Take $N(\frac{\epsilon}{2})\geq 1$ given by the specification property.
Moreover, take $j_{1}=k_{1}=1$, $j_{2}=k_{2}=N(\frac{\epsilon}{2})+1$ and consider preimages $x_{2}$ of $x_{1}$ and $y_{2}$ of $y_{1}$ under $f_{1}^{N(\frac{\epsilon}{2})}$, i.e. $f_{1}^{N(\frac{\epsilon}{2})}(x_{2})=x_{1}$ and $f_{1}^{N(\frac{\epsilon}{2})}(y_{2})=y_{1}$. By applying the specification property for pairs $(x_{1},x_{2})$, $(x_{1},y_{2})$, $(y_{1},x_{2})$ and $(y_{1},y_{2})$
there are points $x_{1,i}\in B(x_{1},\frac{\epsilon}{2})$ and
$y_{1,i}\in B(y_{1},\frac{\epsilon}{2})$, for $i=1, 2$, such that
\begin{center}
$f_{1}^{N(\frac{\epsilon}{2})}(x_{1,1}), f_{1}^{N(\frac{\epsilon}{2})}(y_{1,2})\in B(f_{1}^{N(\frac{\epsilon}{2})}(x_{2}),\frac{\epsilon}{2})=B(x_{1},\frac{\epsilon}{2}),$
\end{center}
\begin{center}
$f_{1}^{N(\frac{\epsilon}{2})}(x_{1,2}), f_{1}^{N(\frac{\epsilon}{2})}(y_{1,1})\in B(f_{1}^{N(\frac{\epsilon}{2})}(y_{2}),\frac{\epsilon}{2})=B(y_{1},\frac{\epsilon}{2})$.
\end{center}
It is clear that the set $\{x_{1,1},x_{1,2},y_{1,1},y_{1,2}\}$
is $(N(\frac{\epsilon}{2}),\epsilon;f_{1,\infty})$-separated. In particular, it follows
that $s_{N(\frac{\epsilon}{2})}(f_{1,\infty};\epsilon)\geq 2^{2}$.

Next, we take $j_{3}=k_{3}=2N(\frac{\epsilon}{2})+1$ and consider preimages $x_{3}$ of $x_{1}$ and $y_{3}$ of $y_{1}$ under $f_{1}^{2N(\frac{\epsilon}{2})}$, i.e. $f_{1}^{2N(\frac{\epsilon}{2})}(x_{3})=x_{1}$ and $f_{1}^{2N(\frac{\epsilon}{2})}(y_{3})=y_{1}$. By applying the specification property for triples $(x_{1},x_{2},x_{3})$, $(x_{1},x_{2},y_{3})$, $(x_{1},y_{2},x_{3})$, $(x_{1},y_{2},y_{3})$, $(y_{1},y_{2},y_{3})$, $(y_{1},y_{2},x_{3})$, $(y_{1},x_{2},y_{3})$ and $(y_{1},x_{2},x_{3})$,
there are the points $x_{1,j}\in B(x_{1},\frac{\epsilon}{2})$ and
$y_{1,j}\in B(y_{1},\frac{\epsilon}{2})$, $j=1, \ldots, 4$, for which the following hold:
\begin{enumerate}
\item [-] $f_{1}^{N(\frac{\epsilon}{2})}(x_{1,1}), f_{1}^{2N(\frac{\epsilon}{2})}(x_{1,1})\in B(x_{1},\frac{\epsilon}{2}),$ $f_{1}^{N(\frac{\epsilon}{2})}(x_{1,4}), \text{and} \ f_{1}^{2N(\frac{\epsilon}{2})}(x_{1,4})\in B(y_{1},\frac{\epsilon}{2});$
\item [-] $f_{1}^{N(\frac{\epsilon}{2})}(x_{1,2})\in B(x_{1},\frac{\epsilon}{2}), \ \text{and} \ f_{1}^{2N(\frac{\epsilon}{2})}(x_{1,2})\in B(y_{1},\frac{\epsilon}{2});$
\item [-] $f_{1}^{N(\frac{\epsilon}{2})}(x_{1,3})\in B(y_{1},\frac{\epsilon}{2}), \ \text{and} \ f_{1}^{2N(\frac{\epsilon}{2})}(x_{1,3})\in B(x_{1},\frac{\epsilon}{2});$
\item [-] $f_{1}^{N(\frac{\epsilon}{2})}(y_{1,1}), f_{1}^{2N(\frac{\epsilon}{2})}(y_{1,1})\in B(y_{1},\frac{\epsilon}{2}),$ $f_{1}^{N(\frac{\epsilon}{2})}(y_{1,4}), \text{and} \ f_{1}^{2N(\frac{\epsilon}{2})}(y_{1,4})\in B(x_{1},\frac{\epsilon}{2});$
\item [-] $f_{1}^{N(\frac{\epsilon}{2})}(y_{1,2})\in B(y_{1},\frac{\epsilon}{2}) \ \text{and} \ f_{1}^{2N(\frac{\epsilon}{2})}(y_{1,2})\in B(x_{1},\frac{\epsilon}{2});$
  \item [-] $f_{1}^{N(\frac{\epsilon}{2})}(y_{1,3})\in B(x_{1},\frac{\epsilon}{2})\ \text{and} \ f_{1}^{2N(\frac{\epsilon}{2})}(y_{1,3})\in B(y_{1},\frac{\epsilon}{2}).$
\end{enumerate}
It is clear that the set $\{x_{1,1},x_{1,2},x_{1,3},x_{1,4},y_{1,1},y_{1,2},y_{1,3},y_{1,4}\}$
is $(2N(\frac{\epsilon}{2}),\epsilon;f_{1,\infty})$-separated. In particular, it follows
that $s_{2N(\frac{\epsilon}{2})}(f_{1,\infty};\epsilon)\geq 2^{3}$.

Now, let $n=dN(\frac{\epsilon}{2})+1$ where $d\in\mathbb{N}$. Take
$j_{1}=k_{1}=1, j_{2}=k_{2}=N(\frac{\epsilon}{2})+1, j_{3}=k_{3}=2N(\frac{\epsilon}{2})+1,\ldots, j_{d}=k_{d}=(d-1)N(\frac{\epsilon}{2})+1, j_{d+1}=k_{d+1}=dN(\frac{\epsilon}{2})+1$
and consider preimages $x_{i}$ of $x_{1}$ and $y_{i}$ of $y_{1}$ under $f_{1}^{(i-1)N(\frac{\epsilon}{2})}$ for $i=2,\ldots,d+1$, i.e. $f_{1}^{(i-1)N(\frac{\epsilon}{2})}(x_{i})=x_{1}$ and $f_{1}^{(i-1)N(\frac{\epsilon}{2})}(y_{i})=y_{1}$. By repeating the previous reasoning for $(d+1)$-tuples in which the $i$th
component choosing from the set $\{x_{i},y_{i}\}$, it follows
that $s_{dN(\frac{\epsilon}{2})}(f_{1,\infty};\epsilon)\geq 2^{d+1}$. Thus,
\begin{eqnarray*}
\limsup_{n\to\infty}\frac{1}{n}\log s_{n}(f_{1,\infty};\epsilon)
&\geq &\limsup_{d\to\infty}\frac{1}{dN(\frac{\epsilon}{2})}\log s_{dN(\frac{\epsilon}{2})}(f_{1,\infty};\epsilon)\\
&\geq & \limsup_{d\to\infty}\frac{1}{dN(\frac{\epsilon}{2})}\log 2^{d+1}=\dfrac{\log 2}{N(\frac{\epsilon}{2})}.
\end{eqnarray*}
This proves that the topological entropy is positive and finishes the proof of the theorem.
\end{proof}
The following corollary is an immediate consequence of Theorem \ref{theorem4} and Theorem \ref{theorem3}, that is a criterion under which an NDS does not have the specification property.
\begin{corollary}\label{corollary3}
Let $(X, f_{1,\infty})$ be an NDS of surjective maps on a compact metric space $X$ without any isolated point. If $h_{\text{top}}(X, f_{i,\infty})=0$ for some $i\geq 1$, then the NDS $(X, f_{1,\infty})$ does not have the specification property.
\end{corollary}
As a direct consequence of Theorem \ref{theorem4} and \cite[Lemma 4.5]{KS} we have the next corollary which says that every NDS of surjective maps on a compact metric space with the specification property is topologically chaotic.
\begin{corollary}\label{corollary1}
Let $(X, f_{1,\infty})$ be an NDS of surjective maps on a compact metric space $X$ without any isolated point. If the NDS $(X, f_{1,\infty})$ satisfies the specification property, then it has positive asymptotical topological entropy. In particular, it is topologically chaotic.
\end{corollary}
As we have seen before, for NDSs of surjective maps on a compact metric space with the specification property,
local neighborhoods reflect the complexity of the entire dynamical system from the viewpoint
of entropy theory and asymptotical topological entropy. Also, by Theorem \ref{theorem4}, every NDS of surjective maps on a compact metric space with the specification property has positive topological entropy. Thus, by Theorem \ref{theorem3}, for NDSs of surjective maps on a compact metric space with the specification property, local neighborhoods have positive topological entropy. More precisely, we have the following corollary.
\begin{corollary}
Let $(X, f_{1,\infty})$ be an NDS of surjective maps on a compact metric space $X$ without any isolated point. If the NDS $(X, f_{1,\infty})$ satisfies the specification property, then $h_{\text{top}}(cl(U), f_{1,\infty})>0$ for any $x\in X$ and any open neighborhood $U$ of $x$.
\end{corollary}
\section{Ruelle-expanding NDSs}\label{section4}
The main aim of this section is to introduce a special class of NDSs having the specification property. More precisely, we introduce uniformly Ruelle-expanding NDSs and provide some sufficient conditions that these systems have the specification property.
\begin{definition}[Ruelle-expanding map]
Based on \cite{DR2}, we say that a continuous onto transformation $f :X\to X$ on a compact metric space $(X,d)$ is a \emph{Ruelle-expanding map} if it is both open and expanding. Recall that $f$ is said to be an \emph{expanding map} if there exist constants $\sigma>1$ and $\rho>0$ such that for every $p\in X$ the image of the ball $B(p, \rho)$ contains a neighborhood of the closure of $B(f(p),\rho)$ and
\begin{equation}\label{exp444}
 d(f(x), f(y))\geq\sigma d(x, y)\ \text{for every} \ x, y \in B(p,\rho).
\end{equation}
\end{definition}
Let $f:X\to X$ be a Ruelle-expanding map on a compact metric space $(X,d)$. Based on \cite{VO}, the restriction of $f$ to each ball $B(p,\rho)$ of radius $\rho$ is injective and its image contains the closure of $B(f(p),\rho)$. Thus, the restriction to $B(p,\rho)\cap f^{-1}(B(f(p),\rho))$ is a homeomorphism onto $B(f(p),\rho)$. We denote by
$$h_{p}:B(f(p),\rho)\to B(p,\rho)$$
its inverse and call it \emph{inverse branch} of $f$ at $p$. It is clear that $h_{p}(f(p))=p$
and $f\circ h_{p}=\text{id}$. The condition (\ref{exp444}) implies that $h_{p}$ is a $\sigma^{-1}$-contraction:
\begin{equation*}
 d(h_{p}(z),h_{p}(w))\leq \sigma^{-1}d(z,w)\ \text{for every} \ z,w \in B(f(p),\rho).
\end{equation*}
The factors $\sigma$ and $\rho$ will be called the \emph{expansion factor} and \emph{injectivity constant} of the  Ruelle-expanding map $f$, respectively.

Note that, one-sided Markov subshifts of finite type, determined by aperiodic square matrices with entries in $\{0,1\}$, and $C^{1}$ expanding maps on compact manifolds are examples of Ruelle-expanding maps. Moreover, if the domain of a Ruelle-expanding map is connected, then it is topologically mixing and topologically exact, see \cite{VO,DR2}.

Now, we introduce a certain class of NDSs that will be studied in the present section.
\begin{definition}[Uniformly Ruelle-expanding NDS]
Let $(X, f_{1,\infty})$ be an NDS of Ruelle-expanding maps $f_{n}$ with expansion factors $\sigma_{n}$ and injectivity constants $\rho_{n}$. We say that the NDS $(X, f_{1,\infty})$ is a \emph{uniformly Ruelle-expanding NDS} if there exist constants $\sigma>1$ and $\rho>0$ such that $\sigma_{n}>\sigma$ and $\rho_{n}>\rho$ for every $n\geq 1$.
The factors $\sigma$ and $\rho$ will be called \emph{uniform expansion factor} and \emph{uniform injectivity constant} of the uniformly Ruelle-expanding NDS $(X, f_{1,\infty})$, respectively.
\end{definition}
In what follows, we consider a uniformly Ruelle-expanding NDS $(X, f_{1,\infty})$ with uniform expansion factor $\sigma$ and injectivity constant $\rho$. By definition, for every $i$, the restriction of $f_i$ to
each ball $B(x, \rho)$ of radius $\rho$ is injective and its image contains the closure of $B(f_i(x), \rho)$.
Thus, the restriction $f_i$ to
$B(x,\rho) \cap f_i^{-1}(B(f_i(x), \rho))$ is a homeomorphism onto $B(f_i(x), \rho)$, we denote by
\begin{equation*}
 h_{i,x}:B(f_i(x), \rho)\to B(x,\rho)
\end{equation*}
the inverse branch of $f_i$ at $x$. It is clear that $h_{i,x}(f_i(x)) = x$ and
$f_i \circ h_{i,x} =id$. By the above statements $h_{i,x}$ is
a $\sigma^{-1}$-contraction:
\begin{equation}\label{eq5}
d(h_{i,x}(z),h_{i,x}(w))\leq\sigma^{-1}d(z,w)\ \ \text{for every}\ \ z,w\in B(f_i(x), \rho).
\end{equation}

More generally, for any $k,n \geq 1$, we call the inverse branch of $f_{k}^{n}$ at $x$ the composition
\begin{equation*}
 h_{k,x}^{n}:= h_{k,x} \circ h_{k+1,f_{k}^{1}(x)} \circ h_{k+2,f_{k}^{2}(x)} \circ\cdots\circ h_{k+n-1,f_{k}^{n-1}(x)}: B(f_{k}^{n}(x), \rho) \to B(x,\rho).
 \end{equation*}
Observe that $h_{k,x}^{n} (f_{k}^{n}(x)) = x$ and $f_{k}^{n} \circ h_{k,x}^{n}= id$. Moreover, for each
$0 \leq j \leq n$ we have
\begin{center}
$f_{k}^{j} \circ h_{k,x}^{n} = h_{k+j,f_{k}^{j}(x)}^{n-j}\ \ \text{and} \ \ h_{k+j,f_{k}^{j}(x)}^{n-j}: B(f_{k}^{n}(x), \rho) \to B(f_{k}^{j}(x),\rho)$.
\end{center}
Hence,
\begin{equation}\label{in}
 d(f_{k}^{j} \circ h_{k,x}^{n} (z), f_{k}^{j} \circ h_{k,x}^{n} (w)) \leq \sigma^{j-n}d(z,w)
\end{equation}
for every $z,w \in B(f_{k}^{n}(x), \rho)$ and every $0 \leq j \leq n$.

In the rest of this section, we provide some sufficient conditions that uniformly Ruelle-expanding NDSs on compact metric spaces satisfy the specification property. Thus, by Theorems \ref{theorem4} and \ref{theorem3}, any uniformly Ruelle-expanding NDS with these conditions on a compact metric space has positive topological entropy and all points are entropy point; in particular, this system is topologically chaotic. We mention that our result about the positivity of topological entropy extends the already result known by \cite{K2}. In \cite{K2} the author considers uniformly expanding NDSs built from $C^{2}$ expanding maps on a compact Riemannian manifold $M$ with uniform bounds on expansion factors and derivatives; however, in this article, we consider the topological version of uniformly expanding NDSs.
\subsection{Specification and topologically exact property}
Our aim of this subsection is to prove that any uniformly Ruelle-expanding NDS with the topologically exact property enjoys the specification property (Theorem \ref{theorem1} ). We need the following auxiliary lemma in the proof of Theorem \ref{theorem1} below.
\begin{lemma}\label{lem0}
Let $(X, f_{1,\infty})$ be a uniformly Ruelle-expanding NDS with uniform expansion factor $\sigma$ and injectivity constant $\rho$. Then for every $x \in X$, $k\in\mathbb{N}, n\geq 0$ and $0< \epsilon \leq \rho$ we have $f_{k}^{n}(B(x,k,n,\epsilon))=B(f_{k}^{n}(x),\epsilon)$, where $B(x,k,n,\epsilon)$ is the dynamical $(n+1)$-ball with initial time $k$ around $x$ of radius $\epsilon$ given by (\ref{01}).
\end{lemma}
\begin{proof}
Let $B(x,k,n,\epsilon)$ be a dynamical $(n+1)$-ball with initial time $k$ around $x$. We prove
that $f_{k}^{n}(B(x,k,n,\epsilon))=B(f_{k}^{n}(x),\epsilon)$. The
inclusion $f_{k}^{n}(B(x,k,n,\epsilon))\subseteq B(f_{k}^{n}(x),\epsilon)$ is an immediate consequence of the definition of dynamical ball. To prove the converse, consider the inverse
branch $h_{k,x}^{n}: B(f_{k}^{n}(x), \rho) \to B(x,\rho)$.
Given any $y \in B(f_{k}^{n}(x),\epsilon)$, let $z = h_{k,x}^{n}(y)$. Then, $f_{k}^{n}(z) = y$ and,
by  relation (\ref{in}) we have
$$d(f_{k}^{j}(z), f_{k}^{j}(x)) \leq \sigma^{j-n}d(f_{k}^{n}(z), f_{k}^{n}(x)) \leq
d(y, f_{k}^{n}(x)) < \epsilon $$
for every $0 \leq j \leq n$. This shows that $z \in B(x, k, n, \epsilon)$ and finishes
the proof of the lemma.
\end{proof}
\begin{definition}[Topologically exact property]
An NDS $(X, f_{1,\infty})$ has the \emph{topologically exact property} if for any $\delta>0$ there is $N=N(\delta)\in\mathbb{N}$ so that $f_{k}^{n}(B(x,\delta))=X$ for every $x\in X$, $k\geq 1$ and $n\geq N$.
\end{definition}
\begin{theorem}\label{theorem1}
Let $(X, f_{1,\infty})$ be a uniformly Ruelle-expanding NDS on a compact metric space $X$ with uniform expansion factor $\sigma$ and injectivity constant $\rho$. If the NDS $(X, f_{1,\infty})$ satisfies the topologically exact property, then it has the specification property.
\end{theorem}
\begin{proof}
The proof of the theorem can be followed from the topologically exact property and the previous lemma. Fix $\delta>0$, without loss of generality we assume that $\delta<\rho$. Let $N=N(\delta)$ be given by the topologically exact property. Now, suppose that the points $x_{1},x_{2},\ldots,x_{s}\in X$ with $s\geq 2$ and a sequence $1=j_{1}\leq k_{1}<j_{2}\leq k_{2}<\cdots<j_{s}\leq k_{s}$ of integers with
$j_{n}-k_{n-1}\geq N$ for $n=2,\ldots,s$ are given. By Lemma \ref{lem0} we have
\begin{equation}\label{eq1}
f_{j_{i}}^{k_{i}-j_{i}}\big(B(f_{1}^{j_{i}-1}(x_{i}),j_{i},k_{i}-j_{i},\delta)\big)=B(f_{j_{i}}^{k_{i}-j_{i}}(f_{1}^{j_{i}-1}(x_{i})),\delta)\ \text{for}\ 1\leq i\leq s
\end{equation}
and by the topologically exact property we have
\begin{equation}\label{eq2}
f_{k_{i}}^{j_{i+1}-k_{i}}\big(B(f_{j_{i}}^{k_{i}-j_{i}}(f_{1}^{j_{i}-1}(x_{i})),\delta)\big)=X\ \text{for}\ i=1,\ldots,s-1.
\end{equation}

Equations (\ref{eq1}) and (\ref{eq2}) implies that
for given $\bar{x}_{s}\in B(f_{1}^{j_{s}-1}(x_{s}),j_{s},k_{s}-j_{s},\delta)$, one
has $\bar{x}_{s}=f_{k_{s-1}}^{j_{s}-k_{s-1}}(\tilde{x}_{s-1})$,
with $\tilde{x}_{s-1}\in B(f_{j_{s-1}}^{k_{s-1}-j_{s-1}}(f_{1}^{j_{s-1}-1}(x_{s-1})),\delta)$, and
then $\bar{x}_{s}=f_{k_{s-1}}^{j_{s}-k_{s-1}}\circ f_{j_{s-1}}^{k_{s-1}-j_{s-1}}(\bar{x}_{s-1})$, for some $\bar{x}_{s-1}\in B(f_{1}^{j_{s-1}-1}(x_{s-1}),j_{s-1},k_{s-1}-j_{s-1},\delta)$. By induction, there
exists $\bar{x}_{1}\in B(f_{1}^{j_{1}-1}(x_{1}),j_{1},k_{1}-j_{1},\delta)$, such that
\begin{equation}\label{eq3}
\bar{x}_{i}=f_{k_{i-1}}^{j_{i}-k_{i-1}}\circ f_{j_{i-1}}^{k_{i-1}-j_{i-1}}\circ\cdots\circ f_{k_{1}}^{j_{2}-k_{1}}\circ f_{j_{1}}^{k_{1}-j_{1}}(\bar{x}_{1})\ \text{for}\ i=2,\ldots,s.
\end{equation}
Now, by equation (\ref{eq3}) it is enough to take $x=\bar{x}_{1}$, then $x$ satisfies the definition of specification property.
\end{proof}
This theorem is related to \cite[Theorem 16]{RV} which deals with $C^{1}$ expanding maps. As a direct consequence of Theorems \ref{theorem1}, \ref{theorem4}, \ref{theorem3} and Corollary \ref{corollary1} one can conclude the following corollary.
\begin{corollary}
Let $(X, f_{1,\infty})$ be a uniformly Ruelle-expanding NDS on a compact metric space $X$. If the NDS $(X, f_{1,\infty})$ satisfies the topologically exact property, then it has positive topological entropy and all points are entropy point. In particular, it is topologically chaotic.
\end{corollary}
In the following proposition we introduce a spacial class of uniformly Ruelle-expanding NDSs with the topologically exact property, and so they satisfy the specification property. An NDS $(X, f_{1,\infty})$ is said to be \emph{eventually periodic} if $f_{\ell+k+n}=f_{\ell+n}$ for some $k,\ell\geq 1$ and every $n\geq 1$.
\begin{proposition}
Every eventually periodic uniformly Ruelle-expanding NDS on a compact connected metric space has the topologically exact property, and so it satisfies the specification property.
\end{proposition}
\begin{proof}
Let $(X, f_{1,\infty})$ be a eventually periodic uniformly Ruelle-expanding NDS on a compact connected metric space $X$, i.e. $f_{\ell+k+n}=f_{\ell+n}$ for some $k,\ell\geq 1$ and every $n\geq 1$. Then, by uniform continuity, it is not hard to see that the finite composition $g:=f_{\ell+1}^{k}=f_{\ell+k}\circ\cdots\circ f_{\ell+2}\circ f_{\ell+1}$ is also Ruelle-expanding. Indeed, $g$ is Ruelle-expanding with expansion factor $\sigma_{g}$ and injectivity constant $\rho_{g}$, where
\begin{equation*}
\sigma_{g}=\min_{\ell+1\leq i\leq\ell+k}\{\sigma_{i}\}\ \ \text{and}\ \ \rho_{g}=\min_{\ell+1\leq i\leq\ell+k}\{\rho_{i}\}
\end{equation*}
and $\sigma_{i}$, $\rho_{i}$ are the expansion factor and injectivity constant of Ruelle-expanding map $f_{i}$, respectively, see \cite[Lemma 6.6]{MCFBRPV} for more details. On the other hand, by \cite[Corollary 11.2.16]{VO}, each Ruelle-expanding map on a compact connected metric space is topologically exact. Hence, Ruelle-expanding map $g$ is topologically exact, i.e. for any $\delta>0$ there is $N_{g,\delta}\geq 1$ such that $g^{n}(B(x,\delta))=X$ for every $x\in X$ and every $n\geq N_{g,\delta}$. So the  eventually periodic uniformly Ruelle-expanding NDS $(X, f_{1,\infty})$ satisfies the topologically exact property; for any $\delta>0$, it is enough to take $N_{\delta}=kN_{g,\delta}+\ell$. Now, by the approach used in Theorem \ref{theorem1}, we deduce that the  eventually periodic uniformly Ruelle-expanding NDS $(X, f_{1,\infty})$ satisfies the specification property.
\end{proof}
\begin{remark}\label{remark000}
Let $M$ be a compact Riemannian manifold and $f : M \to M$ be a $C^{1}$ local diffeomorphism, then $f$ is
said to be $C^1$-\emph{expanding} if there exist $\sigma>1$ and some Riemannian metric on $M$ such that
$\|Df(x)v\|\geq \sigma\|v\|$, for every $x \in M$ and every vector $v$ tangent to $M$ at the point $x$.
On the other hand, every $C^1$-expanding map on a manifold is Ruelle-expanding, see \cite[Example 11.2.1]{VO}.
Now, let $(M,f_{1,\infty})$ be a uniformly expanding NDS composed of $C^{1}$-expanding local diffeomorphisms on a compact connected Riemannian manifold $M$. Then, by applying the approach used in \cite[Lemma 18]{RV} or
\cite[Lemma 11.1.13]{VO}, it follows that the NDS $(M,f_{1,\infty})$ has the topologically exact property; hence, it satisfies the specification property.
\end{remark}
\subsection{Specification and topologically mixing property}
In this subsection we prove that any uniformly Ruelle-expanding NDS enjoys the shadowing property (Proposition \ref{proposition5}). Using this fact, we show that every uniformly Ruelle-expanding NDSs with the topologically mixing property satisfy the specification property (Theorem \ref{theorem00}).
\begin{definition}[Shadowing property]
Fix an NDS $(X, f_{1,\infty})$, then
\begin{itemize}
\item
a finite or an infinite sequence $\{x_{1},x_{2},x_{3},\ldots\}\subseteq X$ is called a $\delta$\emph{-pseudo orbit} for some $\delta>0$, if $d(f_{i}(x_{i}),x_{i+1})<\delta$, for all $i\geq 1$.
\item
NDS $(X, f_{1,\infty})$ has the \emph{shadowing property} if for every $\epsilon>0$ there exists $\delta=\delta(\epsilon)>0$ such that for every $\delta$-pseudo orbit $\{x_{1},x_{2},x_{3},\ldots\}$ there exists $x\in X$ such that for all $i\geq1$, $d(f_{1}^{i-1}(x),x_{i})<\epsilon$. In this case we say that the
point $x\in X$ shadows the sequence $\{x_{1},x_{2},x_{3},\ldots\}$.
\end{itemize}
\end{definition}
\begin{definition}[Expansivity]
An NDS $(X, f_{1,\infty})$ is called \emph{expansive} if there exists $\delta>0$ (called expansivity constant) such that for any $x,y\in X$ with $x\neq y$, $d(f_{1}^{n}(x),f_{1}^{n}(y))>\delta$ for some $n\geq 1$. Equivalently, if for $x,y\in X$, $d(f_{1}^{n}(x),f_{1}^{n}(y))\leq\delta$ for all $n\geq 0$, then $x=y$.
\end{definition}
\begin{lemma}\label{lemma2}
Every uniformly Ruelle-expanding NDS is expansive.
\end{lemma}
\begin{proof}
Let $(X, f_{1,\infty})$ be a uniformly Ruelle-expanding NDS with uniform expansion factor $\sigma$ and injectivity constant $\rho$. Assume that $d(f_{1}^{n}(x),f_{1}^{n}(y))\leq\rho$ for every $n\geq 0$. This implies that
$x=h_{1,y}^{n}(f_{1}^{n}(x))$ for every $n\geq 0$. Then, the relation (\ref{in}) gives that
$$d(x,y)\leq\sigma^{-n}d(f_{1}^{n}(x),f_{1}^{n}(y))\leq\sigma^{-n}\rho.$$
Making $n\to\infty$, we get that $x=y$. So, $\rho$ is a constant of expansivity for uniformly Ruelle-expanding NDS $(X, f_{1,\infty})$.
\end{proof}
It is clear that every orbit is a $\delta$-pseudo orbit, for every $\delta>0$. For uniformly Ruelle-expanding NDSs we have a kind of converse: every pseudo-orbit is uniformly close to (we say that it is \emph{shadowed} by) some true orbit of the NDS, see Proposition \ref{proposition5}.

At the same time, Castro, Rodrigues and Varandas assert that (\cite[Lemma 4.2]{CRVV} and \cite[Lemma 3.1]{CRV}) each sequence of Ruelle-expanding maps $f_n: X_n \to X_{n+1}$ on complete (or locally compact) metric spaces $X_n$ with a uniform contraction rates (along the inverse branches) satisfies the Lipschitz shadowing and shadowing properties; however, we have not found any proof. Here, for completion, we give a proof for shadowing property of uniformly Ruelle-expanding NDSs.
\begin{proposition}\label{proposition5}
Let $(X, f_{1,\infty})$ be a uniformly Ruelle-expanding NDS with uniform expansion factor $\sigma$ and injectivity constant $\rho$. Then, the NDS $(X, f_{1,\infty})$ satisfies the shadowing property. If in the definition of shadowing property $\epsilon$ is small enough, so that $2\epsilon$ is a constant of expansivity for the NDS $(X, f_{1,\infty})$, then the point $x$ is unique.
\end{proposition}
\begin{proof}
Let $\epsilon>0$ be given, without loss of generality we assume that $\epsilon<\rho$. Fix $\delta=\delta(\epsilon)>0$ so that $\sigma^{-1}\epsilon+\delta<\epsilon$. Consider a $\delta$-pseudo orbit $\{x_{1},x_{2},x_{3},\ldots\}\subseteq X$. For each $n\geq 1$, let $h_{n,x_{n}}:B(f_{n}(x_{n}),\rho)\rightarrow B(x_{n},\rho)$ be the inverse branch of $f_{n}$ at $x_{n}$. The relation (\ref{eq5}) ensures that
\begin{equation}\label{eq7}
h_{n,x_{n}}(\text{cl}(B(f_{n}(x_{n}),\epsilon)))\subseteq \text{cl}(B(x_{n},\sigma^{-1}\epsilon))\ \ \text{for every}\ \ n\geq 1.
\end{equation}
Since $d(x_{n},f_{n-1}(x_{n-1}))<\delta$ and $\sigma^{-1}\epsilon+\delta<\epsilon$, it follows that
\begin{equation}\label{eq6}
h_{n,x_{n}}(\text{cl}(B(f_{n}(x_{n}),\epsilon)))\subseteq \text{cl}(B(f_{n-1}(x_{n-1}),\epsilon))\ \ \text{for every}\ \ n\geq 1,
\end{equation}
because for $z\in h_{n,x_{n}}(\text{cl}(B(f_{n}(x_{n}),\epsilon)))$ we have
\begin{center}
$d(z,f_{n-1}(x_{n-1}))\leq d(z,x_{n})+d(x_{n},f_{n-1}(x_{n-1}))\leq\sigma^{-1} \epsilon+\delta<\epsilon.$
\end{center}

We may consider the composition $h^{n}=h_{1,x_{1}}\circ\cdots\circ h_{n-1,x_{n-1}}\circ h_{n,x_{n}}$ of inverse branches, then by (\ref{eq6}), the compact subsets $K_{n}:=h^{n}(\text{cl}(B(f_{n}(x_{n}),\epsilon)))$ are nested. Take $x$ in the intersection. For every $n\geq 1$, we have that $x\in K_{n}$ and so $f_{1}^{n-1}(x)$ belongs to
\begin{center}
$f_{1}^{n-1}(K_{n})=f_{1}^{n-1}\circ h^{n}(\text{cl}(B(f_{n}(x_{n}),\epsilon)))=h_{n,x_{n}}(\text{cl}(B(f_{n}(x_{n}),\epsilon))).$
\end{center}
By this fact and (\ref{eq7}), one has that $d(f_{1}^{n-1}(x),x_{n})\leq\sigma^{-1}\epsilon<\epsilon$ for every $n\geq 1$. Consequently, the NDS $(X, f_{1,\infty})$ satisfies the shadowing property.

Now, let $\epsilon$ be small enough, so that $2\epsilon$ is a constant of expansivity for NDS $(X, f_{1,\infty})$.
If $x^{\prime}$ is another point as in the conclusion of the proposition, then
\begin{center}
$d(f_{1}^{n-1}(x),f_{1}^{n-1}(x^{\prime}))\leq d(f_{1}^{n-1}(x),x_{n})+d(x_{n},f_{1}^{n-1}(x^{\prime}))<2\epsilon\ \ \text{for every}\ \ n\geq 1.$
\end{center}
Since $2\epsilon$ is a constant of expansivity for NDS $(X, f_{1,\infty})$, it follows that $x=x^{\prime}$.
This finishes the proof of the theorem.
\end{proof}
\begin{definition}[Topologically mixing property]
An NDS $(X, f_{1,\infty})$ has the \emph{topologically mixing property} if for any two non-empty open sets $U,V\subseteq X$ there exists $N\geq 1$ so that $f_{k}^{n}(U)\cap V\neq\emptyset$ for every $n\geq N$ and $k\geq 1$.
\end{definition}
Note that in \cite{MR}, the authors introduced some kind of specification for NDSs, the so-called specification property on the iterative way, which differs from Definition \ref{SP} in general, and deduced that topologically mixing property and shadowing imply the specification property on the iterative way \cite[Theorem 2.2]{MR}. In the case that $(X, f_{1,\infty})$ is an NDS of surjective maps on a compact metric space $X$, these two kinds of the specification property are the same. The next theorem is a counterpart of \cite[Theorem 2.2]{MR} to our setting.
\begin{theorem}\label{theorem00}
Let $(X, f_{1,\infty})$ be a uniformly Ruelle-expanding NDS on a compact metric space $X$. If the
NDS $(X, f_{1,\infty})$ satisfies the topologically mixing property, then it has the specification property.
\end{theorem}
As a direct consequence of Theorems \ref{theorem00}, \ref{theorem4}, \ref{theorem3} and Corollary \ref{corollary1} one can conclude the following corollary.
\begin{corollary}
Let $(X, f_{1,\infty})$ be a uniformly Ruelle-expanding NDS on a compact metric space $X$. If the NDS $(X, f_{1,\infty})$ satisfies the topologically mixing property, then it has positive topological entropy and all points are entropy point. In particular, it is topologically chaotic.
\end{corollary}
\begin{remark}
Note that uniformly Ruelle-expanding NDSs with the topologically exact property satisfy the topologically mixing property. Hence, Theorem \ref{theorem1} is a direct consequence of Theorem \ref{theorem00}. But we have given the proof, because its approach is necessary in some applications, see Section \ref{applications}.
\end{remark}
\section{Topological pressure}\label{section5}
For NDSs, the notion of topological pressure was introduced by
Huang $et\ al.$ \cite{HWZ} and Kawan \cite{K2}. Based on \cite{HWZ}, one can define the topological pressure of an NDS $(X, f_{1,\infty})$ analogously to the topological pressure of an autonomous dynamical system
$(X,f)$, for example, see \cite{PW}. Indeed, for $f_{1}=f_{2}=\cdots=f$ we get the classical definition.

Fix an NDS $(X, f_{1,\infty})$ and let $\mathcal{C}(X,\mathbb{R})$ be the space of real-valued continuous functions of $X$. For $\psi\in\mathcal{C}(X,\mathbb{R})$ and $i,n\in\mathbb{N}$, denote
$\Sigma_{j=0}^{n-1}\psi(f_{i}^{j}(x))$ by $S_{i,n}\psi(x)$. Also, for subset $U$ of $X$,
put $S_{i,n}\psi(U):=\sup_{x\in U}S_{i,n}\psi(x)$.
\begin{definition}[Topological pressure]
Following \cite{HWZ}, the \emph{topological pressure} of an NDS $(X, f_{1,\infty})$ with respect to any continuous potential $\psi\in\mathcal{C}(X,\mathbb{R})$ defined by
\begin{equation*}
P_{\text{top}}(f_{1,\infty},\psi):=\lim_{\epsilon\to 0}\limsup_{n\to\infty}\frac{1}{n}\log P_{n}(f_{1,\infty},\psi,\epsilon),
\end{equation*}
where
\begin{equation*}
P_{n}(f_{1,\infty},\psi,\epsilon):=\sup\Big\{\sum_{x\in E}e^{S_{1,n}\psi(x)}: E\ \text{is an}\
(n,\epsilon;f_{1,\infty})\text{-separated set for}\ X\Big\}.
\end{equation*}
\end{definition}
Note that, $P_{\text{top}}(f_{1,\infty},\psi)$ exists but could be $\infty$. Moreover, the topological pressure can also be defined in terms of spanning sets and sequences of open covers. For the sake of completeness, given $\epsilon>0$, take
\begin{equation*}
Q_{n}(f_{1,\infty},\psi,\epsilon):=\inf\Big\{\sum_{x\in F}e^{S_{1,n}\psi(x)}: F\ \text{is an}\
(n,\epsilon;f_{1,\infty})\text{-spanning set for}\ X\Big\}.
\end{equation*}
Also, let $\mathcal{A}$ be an open cover of $X$. Put
\begin{equation*}
q_{n}(f_{1,\infty},\psi,\mathcal{A}):=\inf\Big\{\sum_{B\in\mathcal{B}}\inf_{x\in B} e^{S_{1,n}\psi(x)}:\mathcal{B}\ \text{is a finite subcover of}\ \bigvee_{j=0}^{n-1}f_{1}^{-j}(\mathcal{A})\Big\}
\end{equation*}
and
\begin{equation*}
p_{n}(f_{1,\infty},\psi,\mathcal{A}):=\inf\Big\{\sum_{B\in\mathcal{B}}\sup_{x\in B} e^{S_{1,n}\psi(x)}:\mathcal{B}\ \text{is a finite subcover of}\ \bigvee_{j=0}^{n-1}f_{1}^{-j}(\mathcal{A})\Big\}.
\end{equation*}
Then, by \cite[Proposition 2.1 and Remark 2.3 ]{HWZ}, we have the following proposition.
\begin{proposition}\label{proposition51000}
For NDS $(X, f_{1,\infty})$ and continuous potential $\psi\in\mathcal{C}(X,\mathbb{R})$:
\begin{itemize}
\item[(1)] $P_{\text{top}}(f_{1,\infty},\psi)
=\lim_{\epsilon\to 0}\limsup_{n\to\infty}\frac{1}{n}\log Q_{n}(f_{1,\infty},\psi,\epsilon)$.
\item[(2)] $P_{\text{top}}(f_{1,\infty},\psi)
=\lim_{k\to\infty}\limsup_{n\to\infty}\frac{1}{n}\log q_{n}(f_{1,\infty},\psi,\mathcal{A}_{k})$,
where $\mathcal{A}_{k}$ is a sequence of open covers with $\text{diam}(\mathcal{A}_{k})\to 0$. 
\item[(3)] $P_{\text{top}}(f_{1,\infty},\psi)
=\lim_{k\to\infty}\lim_{n\to\infty}\frac{1}{n}\log p_{n}(f_{1,\infty},\psi,\mathcal{A}_{k})$,
where $\mathcal{A}_{k}$ is a sequence of open covers with $\text{diam}(\mathcal{A}_{k})\to 0$. 
\item[(4)] If $\mathcal{A}$ is an open cover of $X$, then $\lim_{n\to\infty}\frac{1}{n}\log p_{n}(f_{1,\infty},\psi,\mathcal{A})$ exists and equals to $\inf_{n}\dfrac{1}{n}\log p_{n}(f_{1,\infty},\psi,\mathcal{A})$.
\end{itemize}
\end{proposition}
\begin{remark}
We recall that Kawan \cite{K2} has generalized the above notions of topological pressure to general
situation $(X_{1,\infty},f_{1,\infty})$ and defined its measure-theoretic counterpart. Also, he obtained a variational inequality for pressure.
\end{remark}
Now, we fix the NDS $(X,f_{1,\infty})$ and consider the \emph{topological pressure function} $P_{\text{top}}(f_{1,\infty},.)$ as a function in the space $\mathcal{C}(X,\mathbb{R})$, with the norm defined
by $\|\psi\|=\sup\{|\psi(x)|: x\in X\}$. Then, similar to the case of autonomous dynamical systems we have the following theorem about the properties of the topological pressure function. For simplicity, we denote $\limsup_{n\to\infty}\frac{1}{n}\log P_{n}(f_{1,\infty},\psi,\epsilon)$ by $P(f_{1,\infty},\psi,\epsilon)$.
\begin{theorem}\label{proposition4}
Let $(X, f_{1,\infty})$ be an NDS of continuous maps on a compact metric space $X$. If $\psi,\varphi\in\mathcal{C}(X,\mathbb{R})$, $\epsilon>0$ and $c\in\mathbb{R}$ then the following are true.
\begin{itemize}
\item[(1)]
$P_{\text{top}}(f_{1,\infty},0)=h_{\text{top}}(X, f_{1,\infty})$.
\item[(2)]
$\varphi\leq\psi$ implies $P_{\text{top}}(f_{1,\infty},\phi)\leq P_{\text{top}}(f_{1,\infty},\psi)$. In particular,
$h_{\text{top}}(X, f_{1,\infty})+\inf\psi\leq P_{\text{top}}(f_{1,\infty},\psi)\leq h_{\text{top}}(X, f_{1,\infty})+\sup\psi$.
\item[(3)]
$P_{\text{top}}(f_{1,\infty},.)$ is either finite valued or constantly $\infty$.
\item[(4)]
$|P(f_{1,\infty},\psi,\epsilon)-P(f_{1,\infty},\varphi,\epsilon)|\leq\|\psi-\varphi\|$. Therefore, if $P_{\text{top}}(f_{1,\infty},.)<\infty$, then $|P_{\text{top}}(f_{1,\infty},\psi)-P_{\text{top}}(f_{1,\infty},\varphi)|\leq\|\psi-\varphi\|$.
\item[(5)]
$P(f_{1,\infty},.,\epsilon)$ is convex, and so if $P_{\text{top}}(f_{1,\infty},.)<\infty$, then $P_{\text{top}}(f_{1,\infty},.)$ is convex.
\item[(6)]
$P_{\text{top}}(f_{1,\infty},\psi+c)=P_{\text{top}}(f_{1,\infty},\psi)+c$.
\item[(7)]
$P_{\text{top}}(f_{1,\infty},\psi+\varphi)\leq P_{\text{top}}(f_{1,\infty},\psi)+P_{\text{top}}(f_{1,\infty},\varphi)$.
\item[(8)]
$P_{\text{top}}(f_{1,\infty},c\psi)\leq cP_{\text{top}}(f_{1,\infty},\psi)$ if $c\geq 1$, $P_{\text{top}}(f_{1,\infty},c\psi)\geq cP_{\text{top}}(f_{1,\infty},\psi)$ if $c\leq 1$.
\item[(9)]
$|P_{\text{top}}(f_{1,\infty},\psi)|\leq P_{\text{top}}(f_{1,\infty},|\psi|)$.
\end{itemize}
\end{theorem}
\begin{proof}
The proof is similar to the case of autonomous systems (for detailed proof see \cite[Theorem 9.7]{PW} and \cite{VO}). But, for the sake of completeness, we sketch a quick proof of part $(4)$. Indeed, let $(a_{i})$ and $(b_{i})$ be collections of positive real numbers, then it is folklore that we have the following simple inequality
\begin{equation}\label{3333}
\dfrac{\sup a_{i}}{\sup b_{i}}\leq\sup\Big(\dfrac{a_{i}}{b_{i}}\Big).
\end{equation}
Now, by the definition of topological pressure and inequality (\ref{3333}), we deduce that
\begin{eqnarray*}
\dfrac{P_{n}(f_{1,\infty},\psi,\epsilon)}{P_{n}(f_{1,\infty},\varphi,\epsilon)} &=&
\dfrac{\sup\Big\{\sum_{x\in E}e^{S_{1,n}\psi(x)}: E\ \text{is an}\
(n,\epsilon;f_{1,\infty})\text{-separated set for}\ X\Big\}}{\sup\Big\{\sum_{x\in E}e^{S_{1,n}\varphi(x)}: E\ \text{is an}\ (n,\epsilon;f_{1,\infty})\text{-separated set for}\ X\Big\}} \\
&\leq & \sup \Bigg\{\dfrac{\sum_{x\in E}e^{S_{1,n}\psi(x)}}{\sum_{x\in E}e^{S_{1,n}\varphi(x)}}: E\ \text{is an}\ (n,\epsilon;f_{1,\infty})\text{-separated set for}\ X\Bigg\} \\
&\leq & \sup \Bigg\{\max_{x\in E}\dfrac{e^{S_{1,n}\psi(x)}}{e^{S_{1,n}\varphi(x)}}: E\ \text{is an}\ (n,\epsilon;f_{1,\infty})\text{-separated set for}\ X\Bigg\} \\
&\leq & e^{n\|\psi-\varphi\|},
\end{eqnarray*}
which proves the part $(4)$.
\end{proof}
\subsection{Thermodynamics of expansive NDSs}
In this section, we study the thermodynamic properties of expansive NDSs and uniformly Ruelle-expanding NDSs. First, we introduce a special class of continuous potentials and provide another formula to compute the topological pressure of this class of continuous potentials (Proposition \ref{proposition3}). Then we obtain conditions under which the topological entropy and topological pressure of any continuous potential can be computed as a limit at a definite size scale (Theorem \ref{theorem6}). Finally, we prove a strong regularity of the topological pressure function (Theorem \ref{theorem111}).

Given $\epsilon>0$, $i,n\in\mathbb{N}$, we say that an
open cover $\mathcal{U}$ of $X$ is a $(i,n,\epsilon)$-\emph{cover} if any open set $U\in\mathcal{U}$
has $d_{i,n}$-diameter smaller than $\epsilon$, where $d_{i,n}$ is the Bowen-metric introduced in (\ref{eq8}).
To obtain another characterization of the topological pressure using these covers, we need
continuous potentials satisfying a regularity condition. Given $\epsilon>0$, $i,n\in\mathbb{N}$ and
$\psi\in\mathcal{C}(X,\mathbb{R})$ we define
the \emph{variation} of $S_{i,n}\psi$ on dynamical balls of radius $\epsilon$ by
$$\text{Var}_{i,n}(\psi,\epsilon):=\sup_{d_{i,n}(x,y)<\epsilon}|S_{i,n}\psi(x)-S_{i,n}\psi(y)|.$$

We say that potential $\psi$ has \emph{uniform bounded variation on dynamical balls of radius} $\epsilon$ if there
exists $C>0$ so that
\begin{equation}\label{222}
\sup_{n}\text{Var}_{1,n}(\psi,\epsilon)\leq C.
\end{equation}
The potential $\psi$ has the \emph{uniformly bounded variation property} whenever there exists $\epsilon>0$
so that $\psi$ has uniform bounded variation on dynamical balls of radius $\epsilon$. Note that equation (\ref{222}) is the counterpart of Bowen's condition to NDSs, see \cite{W1,CL}.

In the next lemma we prove that H\"{o}lder continuous potentials have uniformly bounded variation property for uniformly Ruelle-expanding NDSs.
\begin{lemma}
Let $(X, f_{1,\infty})$ be a uniformly Ruelle-expanding NDS with uniform expansion factor $\sigma$ and injectivity constant $\rho$. Then any H\"{o}lder continuous potential $\psi$ satisfies the uniformly bounded variation property.
\end{lemma}
\begin{proof}
Let $\psi$ be a $(K,\alpha)$-H\"{o}lder continuous potential, i.e. $|\psi(x)-\psi(y)|\leq K d(x,y)^{\alpha}$ for all $x,y\in X$. Then by relation (\ref{in}) for any $0<\epsilon<\rho$, $n\in\mathbb{N}$ and $x,y\in X$ with $d_{1,n}(x,y)<\epsilon$ we have
\begin{eqnarray*}
|S_{1,n}\psi(x)-S_{1,n}\psi(y)|
&=& \Bigg|\sum_{j=0}^{n-1}\psi(f_{1}^{j}(x))-\sum_{j=0}^{n-1}\psi(f_{1}^{j}(y))\Bigg|
\leq \sum_{j=0}^{n-1}|\psi(f_{1}^{j}(x))-\psi(f_{1}^{j}(y))|\\
&\leq & \sum_{j=0}^{n-1}Kd(f_{1}^{j}(x),f_{1}^{j}(y))^{\alpha}\leq\sum_{j=0}^{n-1}K
\sigma^{(j-n)\alpha} d(f_{1}^{n}(x),f_{1}^{n}(y))^{\alpha}\\
&=& \dfrac{K}{1-\sigma^{-\alpha}}\epsilon^{\alpha}.
\end{eqnarray*}
Hence, it is enough to take $C=\dfrac{K}{1-\sigma^{-\alpha}}\epsilon^{\alpha}$ in relation (\ref{222}), as we wanted to prove.
\end{proof}
In the following proposition, we provide another formula for the topological pressure of an NDS for the class of continuous potentials having the uniformly bounded variation property.
\begin{proposition}\label{proposition3}
Let $(X, f_{1,\infty})$ be an NDS and $\psi:X\to\mathbb{R}$ be a continuous potential with the uniformly bounded variation property. Then
\begin{equation*}
P_{\text{top}}(f_{1,\infty},\psi)=\lim_{\epsilon\to 0}\limsup_{n\to\infty}\frac{1}{n}\log \Big(\inf_{\mathcal{U}}\sum_{U\in\mathcal{U}}e^{S_{1,n}\psi(U)}\Big),
\end{equation*}
where the infimum is taken over the $(1,n,\epsilon)$-covers $\mathcal{U}$ of $X$.
\end{proposition}
\begin{proof}
By hypothesis let $\epsilon_{0}>0$ be so that $\psi$ has the uniform bounded variation on dynamical balls
of radius $\epsilon_0$. Take $\epsilon>0$ and $n\in\mathbb{N}$, without loss of generality we assume that $2\epsilon<\epsilon_{0}$.
For simplicity, we denote $\inf_{\mathcal{U}}\sum_{U\in\mathcal{U}}e^{S_{1,n}\psi(U)}$ by
$C_{n}(f_{1,\infty},\psi,\epsilon)$, where the infimum is taken over the $(1,n,\epsilon)$-covers $\mathcal{U}$ 
of $X$. Given an $(n,\epsilon;f_{1,\infty})$-maximal separated set $E$, it follows 
that $\mathcal{U}:=\{B(x,1,n,\epsilon)\}_{x\in E}$ is a
$(1,n,2\epsilon)$-cover. By the uniformly bounded variation property for some constant $C>0$, we deduce that
\begin{equation*}
S_{1,n}\psi(B(x,1,n,\epsilon))=\sup_{z\in B(x,1,n,\epsilon)}S_{1,n}\psi(z)\leq S_{1,n}\psi(x)+C.
\end{equation*}
Consequently,
\begin{equation}\label{eq9}
\limsup_{n\to\infty}\frac{1}{n}\log C_{n}(f_{1,\infty},\psi,2\epsilon)\leq\limsup_{n\to\infty}\frac{1}{n}\log P_{n}(f_{1,\infty},\psi,\epsilon).
\end{equation}
On the other hand, if $\mathcal{U}$ is a $(1,n,\epsilon)$-cover, then for
any $(n,\epsilon;f_{1,\infty})$-separated set $E$ we have
$\mathcal{N}(E)\leq\mathcal{N}(\mathcal{U})$, since the diameter of any $U\in\mathcal{U}$ in
the metric $d_{1,n}$ is less than $\epsilon$. By the uniformly bounded variation property, we get that
\begin{equation}\label{eq10}
\limsup_{n\to\infty}\frac{1}{n}\log P_{n}(f_{1,\infty},\psi,\epsilon)\leq\limsup_{n\to\infty}\frac{1}{n}\log C_{n}(f_{1,\infty},\psi,\epsilon).
\end{equation}
Now, combining equations (\ref{eq9}) and (\ref{eq10}), we have
\begin{eqnarray*}
\limsup_{n\to\infty}\frac{1}{n}\log P_{n}(f_{1,\infty},\psi,\epsilon)
&\leq & \limsup_{n\to\infty}\frac{1}{n}\log C_{n}(f_{1,\infty},\psi,\epsilon)\\
&\leq & \limsup_{n\to\infty}\frac{1}{n}\log P_{n}(f_{1,\infty},\psi,\frac{\epsilon}{2}).
\end{eqnarray*}
This finishes the proof of the proposition.
\end{proof}
By the previous results, the topological pressure of an NDS can be computed as the limiting complexity of the NDS as the size scale $\epsilon$ approaches zero. In what follows, we will be mostly interested in providing conditions for the topological pressure of an NDS to be computed as a limit at a definite size scale. Hence, we begin with the following definition.
\begin{definition}[$\ast$-expansivity]
An NDS $(X, f_{1,\infty})$ is called \emph{strongly} $\delta$-\emph{expansive} for some $\delta>0$ if for any $\gamma>0$ and any $x,y\in X$ with $d(x,y)\geq\gamma$, there exists $k_{0}\geq1$ (depending on $\gamma$) such that $d_{i,n}(x,y)>\delta$ for each $i,n\in\mathbb{N}$ with $n\geq k_{0}$. Also, an NDS is said to be $\ast$-expansive if it is strongly $\delta$-expansive for some $\delta>0$.
\end{definition}
Now, we prove that the topological entropy and topological pressure of any continuous potential of an $\ast$-expansive NDS can be computed as the topological complexity that is observable at a definite size scale. More precisely, we
have the following theorem.
\begin{theorem}\label{theorem6}
Let $(X, f_{1,\infty})$ be a strongly $\delta$-expansive NDS for some $\delta>0$. Then, for every continuous potential $\psi\in\mathcal{C}(X,\mathbb{R})$ and every $0<\epsilon<\delta$,
\begin{equation*}
P_{\text{top}}(f_{1,\infty},\psi)=\limsup_{n\to\infty}\frac{1}{n}\log P_{n}(f_{1,\infty},\psi,\epsilon).
\end{equation*}
In particular,
\begin{equation*}
h_{\text{top}}(X,f_{1,\infty})=\limsup_{n\to\infty}\frac{1}{n}\log s_{n}(f_{1,\infty};\epsilon).
\end{equation*}
\end{theorem}
\begin{proof}
Since $X$ is compact and $\psi$ is continuous, without
loss of generality, we assume that $\psi$ is non-negative. Fix $\gamma$ and $\epsilon$
with $0<\gamma<\epsilon<\delta$. We use the definition of topological pressure in terms of separated sets to get the following inequality
\begin{equation*}
\limsup_{n\to\infty}\frac{1}{n}\log P_{n}(f_{1,\infty},\psi,\gamma)\geq\limsup_{n\to\infty}\frac{1}{n}\log P_{n}(f_{1,\infty},\psi,\epsilon).
\end{equation*}
Hence, it suffices to show that
$$\limsup_{n\to\infty}\frac{1}{n}\log P_{n}(f_{1,\infty},\psi,\gamma)\leq\limsup_{n\to\infty}\frac{1}{n}\log P_{n}(f_{1,\infty},\psi,\epsilon).$$

By the definition of strongly $\delta$-expansivity, for any two distinct
points $x,y\in X$ with $d(x,y)\geq\gamma$, there exists $k_{0}\geq1$ (depending on $\gamma$) such
that $d_{i,n}(x,y)>\delta$ for each $i,n\in\mathbb{N}$ with $n\geq k_{0}$. Take $n\geq 1$ and $k\geq k_{0}$.
Given any $(n,\gamma;f_{1,\infty})$-separated set $E$, we claim
that the set $E$ is $(n+k,\epsilon;f_{1,\infty})$-separated.
In fact, given $x,y\in E$ there exists a $0\leq j\leq n$ so that $d(f_{1}^{j}(x),f_{1}^{j}(y))>\gamma$.
Using that $n+k-j\geq k_{0}$ and definition of strongly $\delta$-expansivity, it follows
that $d_{j+1, n+k-j}(f_{1}^{j}(x),f_{1}^{j}(y))>\delta>\epsilon$.
This implies that $d_{1, n+k}(x,y)>\epsilon$. Hence, $E$ is $(n+k,\epsilon;f_{1,\infty})$-separated which completes the proof of the claim. Now, using that $\psi$ is non-negative, we can conclude that
\begin{center}
$e^{S_{1,n+k}\psi(x)}=e^{S_{1,n}\psi(x)}e^{S_{n+1,k}\psi(f_{1}^{n}(x))}\geq
e^{S_{1,n}\psi(x)}$.
\end{center}
Consequently,
\begin{center}
$P_{n}(f_{1,\infty},\psi,\gamma)\leq P_{n+k}(f_{1,\infty},\psi,\epsilon)$.
\end{center}
Hence
\begin{eqnarray*}
\limsup_{n\to\infty}\frac{1}{n}\log P_{n}(f_{1,\infty},\psi,\gamma)
&\leq &\limsup_{n\to\infty}\frac{1}{n+k}\log P_{n+k}(f_{1,\infty},\psi,\epsilon)\\
&\leq & \limsup_{n\to\infty}\frac{1}{n}\log P_{n}(f_{1,\infty},\psi,\epsilon),
\end{eqnarray*}
which implies that
\begin{equation*}
P_{\text{top}}(f_{1,\infty},\psi)=\limsup_{n\to\infty}\frac{1}{n}\log P_{n}(f_{1,\infty},\psi,\epsilon).
\end{equation*}
The second part is a direct consequence of part (1) of Theorem \ref{proposition4}. This finishes the proof of the theorem.
\end{proof}
\begin{remark}\label{remark512}
We observe that the conclusion of Theorem \ref{theorem6} also holds if we consider open covers instead of
separated sets. More precisely, assume that NDS $(X, f_{1,\infty})$ is strongly $\delta$-expansive for some $\delta>0$. Then, for every continuous potential $\psi:X\to\mathbb{R}$ with the uniform bounded variation
on dynamical balls of radius $\epsilon_0$ and every $0<\epsilon<\min\{\delta,\epsilon_{0}/2\}$,
\begin{equation*}
P_{\text{top}}(f_{1,\infty},\psi)=\limsup_{n\to\infty}\frac{1}{n}\log \Big(\inf_{\mathcal{U}}\sum_{U\in\mathcal{U}}e^{S_{1,n}\psi(U)}\Big),
\end{equation*}
where the infimum is taken over the $(1,n,\epsilon)$-covers $\mathcal{U}$ of $X$.
\end{remark}
In the following lemma we introduce a special class of NDSs with $\ast$-expansive property.
\begin{lemma}\label{lemma512}
Let $(X, f_{1,\infty})$ be a uniformly Ruelle-expanding NDS with uniform expansion factor $\sigma$ and injectivity constant $\rho$. Then the NDS $(X, f_{1,\infty})$ is $\ast$-expansive.
\end{lemma}
\begin{proof}
By assumption, all inverse branches of $f_{n}$, for each $n\geq 1$, are defined in balls of radius $\rho$ and
they are $\sigma^{-1}$ contraction. For given $\gamma>0$, take $k_{0}\geq 1$ (depending
on $\gamma$) so that $\sigma^{-k_{0}}\rho<\gamma$. We claim that  for any $x,y\in X$ with $d(x,y)>\gamma$
and $i,n\in\mathbb{N}$ with $n\geq k_{0}$ we have $d_{i,n}(x,y)>\rho$, i.e. NDS $(X, f_{1,\infty})$ is
strongly $\rho$-expansive. Assume, by contradiction, that there exist $i,n\in\mathbb{N}$ with $n\geq k_{0}$ such that $d_{i,n}(x,y)<\rho$. Then, by relation (\ref{in}) we have $d_{i,j}(x,y)\leq \sigma^{j-n}d_{i,n}(x,y)$ for every $0\leq j\leq n$. So $d(x,y)\leq \sigma^{-n}d_{i,n}(x,y)<\sigma^{-n}\rho\leq\sigma^{-k_{0}}\rho<\gamma$, which is a contradiction. This finishes the proof of the lemma.
\end{proof}
The next result is a consequence of the Theorem \ref{theorem6}, Remark \ref{remark512} and Lemma \ref{lemma512}.
\begin{corollary}
Let $(X, f_{1,\infty})$ be a uniformly Ruelle-expanding NDS with uniform expansion factor $\sigma$ and injectivity constant $\rho$. Then, for every continuous potential $\psi$ and every $0<\epsilon<\rho$,
\begin{equation*}
P_{\text{top}}(f_{1,\infty},\psi)=\limsup_{n\to\infty}\frac{1}{n}\log P_{n}(f_{1,\infty},\psi,\epsilon)\ \text{and}\ h_{\text{top}}(X,f_{1,\infty})=\limsup_{_{n\to\infty}}\dfrac{1}{n}\log s_{n}(f_{1,\infty};\epsilon).
\end{equation*}
Additionally, for every continuous potential $\psi:X\to\mathbb{R}$ with the uniform bounded variation
on dynamical balls of radius $\epsilon_0$ and every $0<\epsilon<\min\{\rho,\epsilon_{0}/2\}$,
\begin{equation*}
P_{\text{top}}(f_{1,\infty},\psi)=\limsup_{n\to\infty}\frac{1}{n}\log \Big(\inf_{\mathcal{U}}\sum_{U\in\mathcal{U}}e^{S_{1,n}\psi(U)}\Big),
\end{equation*}
where the infimum is taken over the $(1,n,\epsilon)$-covers $\mathcal{U}$ of $X$.
\end{corollary}
In the rest of this section, we study the Lipschitz regularity of the topological pressure function. Hence, we begin with the following proposition that provide a condition under which the topological pressure can be computed as a limit.
\begin{proposition}\label{pro999}
Let $(X, f_{1,\infty})$ be an NDS and $\psi:X\to\mathbb{R}$ be a continuous potential. Given $\epsilon>0$, the limit superior
\begin{equation*}
\limsup_{n\to\infty}\frac{1}{n}\log \Big(\inf_{\mathcal{U}}\sum_{U\in\mathcal{U}}e^{S_{1,n}\psi(U)}\Big)
\end{equation*}
is indeed a limit, where the infimum is taken over the $(1,n,\epsilon)$-covers $\mathcal{U}$ of $X$.
\end{proposition}
\begin{proof}
This is a direct consequence of part $(4)$ of Proposition \ref{proposition51000}. Indeed, for every $\epsilon>0$, it is enough to take $\mathcal{A}:=\{B(x,\epsilon): x\in X\}$, where $B(x,\epsilon)$ is the open ball with center $x$ and radius $\epsilon$.
\end{proof}
\begin{theorem}\label{theorem111}
Let $(X, f_{1,\infty})$ be a strongly $\delta$-expansive NDS for some $\delta>0$, and let $\psi: X \to \mathbb{R}$ be a continuous potential with the uniformly bounded variation property. Then the following properties hold:
\begin{enumerate}
\item [(1)] the pressure function $t\mapsto P_{\text{top}}(f_{1,\infty},t\psi)$ is an uniform limit of differentiable maps;
\item [(2)] $t\mapsto P_{\text{top}}(f_{1,\infty},t\psi)$ is differentiable Lebesgue-almost everywhere.
\end{enumerate}
\end{theorem}
\begin{proof}
To prove we apply the approach used by Rodrigues and Varandas \cite[Theorem 27]{RV} which deals with finitely generated semigroup actions to our setting.
By hypothesis let $\epsilon_{0}>0$ be so that $\psi$ has uniform bounded variation on dynamical balls of radius $\epsilon_0$. By Remark \ref{remark512}, for any $0<\epsilon<\min\{\delta, \epsilon_0/2\}$ we have
\begin{equation*}
P_{\text{top}}(f_{1,\infty},\psi)=\limsup_{n\to\infty}\frac{1}{n}\log \Big(\inf_{\mathcal{U}}\sum_{U\in\mathcal{U}}e^{S_{1,n}\psi(U)}\Big),
\end{equation*}
where the infimum is taken over the $(1,n,\epsilon)$-covers $\mathcal{U}$ of $X$. Also, by Proposition \ref{pro999}, the right hand side of above equality is actually a true limit. Thus, for any $t \in \mathbb{R}$ we have that
\begin{equation*}
P_{\text{top}}(f_{1,\infty},t\psi)=\lim_{n\to\infty}\frac{1}{n}\log \Big(\inf_{\mathcal{U}}\sum_{U\in\mathcal{U}}e^{t S_{1,n}\psi(U)}\Big),
\end{equation*}
where the infimum is taken over the $(1,n,\epsilon)$-covers $\mathcal{U}$ of $X$ for
any $0<\epsilon<\min\{\delta, \epsilon_0/2\}$. It means that the 
map $t \mapsto P_{\text{top}}(f_{1,\infty}, t\psi)$
 is a pointwise limit of real analytic functions. We show that the convergence is uniform. To prove this we will prove that the sequence of real functions $\big(P_n(t\psi)\big)_{n\geq 1}$ defined by
\begin{equation*}
 t\mapsto P_n(t\psi):= \frac{1}{n}\log C_n(f_{1,\infty}, t\psi, \epsilon),
\end{equation*}
where
\begin{equation*}
 C_n(f_{1,\infty}, t\psi, \epsilon)=\inf_{\mathcal{U}}\sum_{U\in\mathcal{U}}e^{t S_{1,n}\psi(U)}
\end{equation*}
and the infimum is taken over the $(1,n,\epsilon)$-covers $\mathcal{U}$ of $X$ is equicontinuous in compact intervals. This means that
given $\epsilon > 0$ there exists $\delta^{\prime}> 0$ such that if $|t_1 - t_2|<\delta^{\prime}$
then $|P_n(t_1 \psi) - P_n(t_2 \psi)|<\epsilon$, for every $n\in\mathbb{N}$.
In what follows assume $\epsilon>0$ is fixed and let $0<\delta^{\prime}<\epsilon/ \|\psi\|$.
Given $t_1,t_2$ arbitrary with $|t_1 - t_2|<\delta^{\prime}$, we get
\begin{eqnarray*}
|P_n(t_1 \psi) - P_n(t_2 \psi)| &=& \frac{1}{n}\log \Bigg[\frac{\inf_{\mathcal{U}}\sum_{U\in\mathcal{U}}e^{t_1 S_{1,n}\psi(U)}}{\inf_{\mathcal{U}}\sum_{U\in\mathcal{U}}e^{t_2 S_{1,n}\psi(U)}}\Bigg] \\
&\leq&  \frac{1}{n}\log \Bigg[\frac{e^{n\delta^{\prime}\|\psi\|}\inf_{\mathcal{U}}\sum_{U\in\mathcal{U}}e^{t_2 S_{1,n}\psi(U)}}{\inf_{\mathcal{U}}\sum_{U\in\mathcal{U}}e^{t_2 S_{1,n}\psi(U)}}\Bigg] \\
&=& \delta^{\prime}\|\psi\|<\epsilon.
\end{eqnarray*}
Thus the sequence $\big(P_n(t\psi)\big)_{n\geq 1}$ is equicontinuous. Since $\big(P_n(t\psi)\big)_{n\geq 1}$ converges pointwise, we have that the sequence converges uniformly on compact intervals and so the map
$t \mapsto P_{\text{top}}(f_{1,\infty}, t\psi)$ is continuous. Furthermore, for any $n\in \mathbb{N}$ the function $t \mapsto P_n(\psi + t\varphi)$ is differentiable and
\begin{equation*}
\Big |\frac{dP_n(\psi + t\varphi)}{dt} \Big |=\frac{1}{C_n(f_{1,\infty}, \psi+t\varphi, \epsilon)}\frac{1}{n}\Big( \inf_{\mathcal{U}}\sum_{U\in\mathcal{U}}S_{1,n}\varphi(U)e^{ S_{1,n}(\psi + t\varphi)(U)}\Big)
\end{equation*}
is bounded from above by $\|\varphi\|$, where $\varphi: X \to \mathbb{R}$ is a continuous potential with the uniformly bounded variation property and the infimum is taken over the $(1,n,\epsilon)$-covers $\mathcal{U}$ of $X$. This proves property (1). Moreover, by the mean value inequality
\begin{equation*}
  |P_n(\psi) - P_n(\varphi)| \leq \sup_{0\leq t\leq 1} \Big | \frac{dP_n (\psi +t(\varphi - \psi) )}{dt}\Big | \leq \|\psi - \varphi\|.
\end{equation*}
Taking $n \to \infty$ we get
that $|P_{\text{top}}(f_{1,\infty}, \psi) - P_{\text{top}}(f_{1,\infty}, \varphi)| \leq \|\psi - \varphi\|$ (note that, this is another proof of part $(4)$ of Theorem \ref{proposition4}) and so the pressure function $P_{\text{top}}(f_{1,\infty},.)$ acting on the space of potentials with uniformly bounded variation property is Lipschitz continuous with Lipschitz constant equal to one. On the other hand, Lipschitz functions on the real line are differentiable almost everywhere. Using these facts, we deduce that $t \mapsto P_{\text{top}}(f_{1,\infty}, t\psi)$ is Lebesgue-almost everywhere differentiable, which completes the proof of the theorem.
\end{proof}
The next corollary is a consequence of the Theorem \ref{theorem111} and Lemma \ref{lemma512}.
\begin{corollary}
Let $(X, f_{1,\infty})$ be a uniformly Ruelle-expanding NDS and $\psi: X \to \mathbb{R}$ be a continuous potential with the uniformly bounded variation property. Then the following properties hold:
\begin{enumerate}
\item [(1)] the pressure function $t\mapsto P_{\text{top}}(f_{1,\infty},t\psi)$ is an uniform limit of differentiable maps;
\item [(2)] $t\mapsto P_{\text{top}}(f_{1,\infty},t\psi)$ is differentiable Lebesgue-almost everywhere.
\end{enumerate}
\end{corollary}
\section{Applications}\label{applications}
In this section, we provide several examples of uniformly Ruelle-expanding NDSs that satisfies the specification 
and $\ast$-expansive properties (Examples \ref{example1}, \ref{example2}, \ref{example3} and \ref{example4}). Consequently, they have positive topological entropy and all points are entropy point. In particular, they are topologically chaotic. Moreover, we give an example illustrates that no element of an NDS need to have uniform expansion factor for the NDS to have the specification property (Example \ref{example5}). Finally, as an application of Corollary \ref{corollary3}, we provide an example which does not have the specification property (Example \ref{example6}).
\begin{example}\label{example1}
Let $X=\{1,2,\ldots,d\}$ and $A=(a_{ij})$ be a $d$ by $d$ transition matrix, i.e. a square matrix of dimension $d\geq 2$ such that (i) $a_{ij}=0,1$ for all $i$ and $j$, (ii) $\sum_{j}a_{ij}\geq 1$ for all $i$, and (iii) $\sum_{i}a_{ij}\geq 1$ for all $j$. Consider the subset $\Sigma_{A}$ of $\Sigma=X^{\mathbb{N}}$ consisting of all the sequences $(x_{n})_{n}\in\Sigma$ that are $A$-admissible, meaning that $a_{x_{n}x_{n+1}}=1$ for every $n\in\mathbb{N}$. It is clear that $\Sigma_{A}$ is invariant under the shift map $\sigma:\Sigma\to\Sigma$ (defined by $(\sigma\textbf{x})_{n}=x_{n+1}$ for each $\textbf{x}=(x_{n})_{n}\in\Sigma$) in the sense that $\sigma(\Sigma_{A})=\Sigma_{A}$. Note also that $\Sigma_{A}$ is closed in $\Sigma$, and so it is a compact metric space. The restriction $\sigma_{A}:\Sigma_{A}\to\Sigma_{A}$ of the shift map $\sigma:\Sigma\to\Sigma$ to this invariant set is called \emph{one-sided subshift of finite type} associated with $A$. Consider in $\Sigma_{A}$ the distance defined by
\begin{equation*}
d((x_{n})_{n},(y_{n})_{n}):=2^{-N},\ \ \text{where}\ \ N:=\inf\{n\in\mathbb{N}:x_{n}\neq y_{n}\}.
\end{equation*}
Then, $\sigma_{A}^{m}$ is an expanding map for every $m\geq 1$, see \cite[Example 11.2.4]{VO}. 

Now, let $A$ be an eventually positive transition matrix, i.e. there exists a $k\geq 1$ which is independent of $i$ and $j$ such that $(A^{k})_{ij}>0$ for all $i$ and $j$, and let $B$ be a non-empty finite set of positive integers. Set $\mathcal{B}:=\{\sigma_{A}^{m}:m\in B\}$. Then, each NDS $(\Sigma_{A},f_{1,\infty})$ with $f_{n}\in\mathcal{B}$ is a uniformly Ruelle-expanding NDS, and so it has the shadowing property. On the other hand, the NDS $(\Sigma_{A},f_{1,\infty})$ satisfies the topologically mixing property, because $\sigma_{A}$ as an autonomous dynamical system is topologically mixing on $\Sigma_{A}$ and $f_{1}^{n}=\sigma_{A}^{s_{n}}$ for some $s_{n}\geq 1$ and all $n\geq 1$. Hence, by Theorem \ref{theorem00} and Lemma \ref{lemma512}, the NDS $(\Sigma_{A},f_{1,\infty})$ satisfies the specification and $\ast$-expansive properties.
\end{example}
\begin{example}\label{example2}
Let $f_{A}:\mathbb{T}^{d}\to\mathbb{T}^{d}$ be the linear endomorphism of the torus $\mathbb{T}^{d}=\mathbb{R}^{d}/\mathbb{Z}^{d}$ induced by
some matrix $A$ with integer coefficients and determinant different from zero. Assume that all the
eigenvalues $\lambda_{1},\lambda_{2},\ldots,\lambda_{d}$ of $A$ are larger than $1$ in
absolute value. Then, given any $1<\sigma<\inf_{i}|\lambda_{i}|$, there exists an inner product
in $\mathbb{R}^{d}$ relative to which $||Av||\geq\sigma ||v||$ for every $v\in\mathbb{R}^{d}$. This shows that the transformation $f_{A}$ is expanding, see \cite[Example 11.1.1]{VO}.

Now, let $\mathcal{A}$ be a non-empty finite set of different matrices enjoying the above conditions. Then, by Remark \ref{remark000}, each NDS $(\mathbb{T}^{d},f_{1,\infty})$ with $f_{n}\in\{f_{A}:A\in\mathcal{A}\}$ is a uniformly Ruelle-expanding NDS with the topologically exact property. Hence, by Theorem \ref{theorem1} and Lemma \ref{lemma512}, the NDS $(\mathbb{T}^{d},f_{1,\infty})$ enjoys the specification and $\ast$-expansive properties.
\end{example}
\begin{example}\label{example3}
Let $A$ be a non-empty finite set of positive integers $k>1$ and $S^{1}=\mathbb{R}/\mathbb{Z}$. Consider
the set $\mathcal{A}=\{f_{k}:S^{1}\to S^{1}: f_{k}(x)=kx\ \text{(mod 1)},\ k\in A\}$. Then, by Remark \ref{remark000}, each NDS $(S^{1},f_{1,\infty})$ with $f_{n}\in\mathcal{A}$ is a uniformly Ruelle-expanding NDS with the topologically exact property. Hence, by Theorem \ref{theorem1} and Lemma \ref{lemma512}, the NDS $(S^{1},f_{1,\infty})$ satisfies the specification and $\ast$-expansive properties. 
\end{example}
\begin{example}\label{example4}
Consider the NDS $(S^{1},f_{1,\infty})$ on the unit circle $S^{1}$ in which $f_{n}:S^{1}\to S^{1}$ defined by $f_{n}(e^{i\theta}):=e^{i\frac{2n+1}{n}\theta}$. It is clear that the NDS $(S^{1},f_{1,\infty})$ is a uniformly Ruelle-expanding NDS and satisfies the topologically exact property (note that $S^1$ is a geodesic space). Hence, by Theorem \ref{theorem1} and Lemma \ref{lemma512}, the NDS $(S^{1},f_{1,\infty})$ satisfies the specification and $\ast$-expansive properties.
\end{example}
The next example illustrates that no element of a NDS need to have uniform expansion factor for the NDS to have 	the specification property.
\begin{example}\label{example5}
For positive constant $0<\alpha<1$ the Pomeau-Manneville map $f_\alpha:[0,1]\to[0,1]$ given by
 \begin{equation*}
f_\alpha(x)=
\begin{cases}
x+2^\alpha x^{1+\alpha} & \mbox{if}\ 0 \leq x \leq 1/2,\\
2x-1 & \mbox{if}\ 1/2 < x \leq 1.
\end{cases}
\end{equation*}
Although $f_{\alpha}$ is not continuous, it induces a continuous and topologically mixing circle map $\tilde{f}_{\alpha}$ taking $S^{1}=[0,1]/\sim$ with the identification $0\sim 1$. 

Now, let us take $0 <\beta<1$ and the family of real numbers $\{\alpha_{n}: 0 <\beta< \alpha_{n}<1\}$. 
Let $(S^1,f_{1,\infty})$ be an NDS in which $f_{n}:=\tilde{f}_{\alpha_{n}}$ for all $n\geq 1$. Note that no element of the sequence $f_{1,\infty}$ is an expanding map.
We claim that the NDS $(S^1,f_{1,\infty})$ satisfies the specification property. First, we observe that for every $x\in S^{1}$, $\epsilon>0$, $k\in\mathbb{N}$ and $n\geq 0$ the $(n+1)$-dynamical ball $B(x,k,n,\epsilon)$ satisfes $f_{k}^{n}(B(x,k,n,\epsilon))=B(f_{k}^{n}(x),\epsilon)$. Second, although $f_{n}$ is not uniformly expanding, it enjoys the following scaling property: given $\delta > 0$, $\text{diam}(f_{n}([0,\delta]))\geq \frac{\delta}{2}+\frac{\delta}{2}[1+(1+\beta)\delta^\beta]=c_\delta \text{diam}([0,\delta])$
and $\text{diam}(f_{n}(I))\geq \sigma_\delta \text{diam}(I)$ for every ball $I \subset S^1$
of diameter larger or equal to $\delta$, where $c_\delta:=(1+\delta(1+\beta)\delta^\beta)>1$ and $\sigma_\delta >1$ (depending on $\delta$). Here, we use
$f^{\prime}_{\alpha_{n}}(x)\geq (1+(1+\beta)2^\beta x^\beta)\geq (1+(1+\beta)\delta^\beta)$ for 
every $x \in [\frac{\delta}{2},\frac{1}{2}]$ and $f^{\prime}_{\alpha_{n}}(x)=2$ for every $x \in (\frac{1}{2},1]$, see \cite[Example 32]{RV}. Using the previous expression recursively, we deduce that there exists $N_\delta>0$ such that for each $k\geq 1$ and $n \geq N_\delta$ one has that $f_{k}^{n}(B(x,\delta))=S^1$, for each $x \in S^1$. This means that the NDS $(S^1,f_{1,\infty})$ has the topologically exact property. Thus, we can apply the approach used in the proof of Theorem \ref{theorem1} to conclude the NDS $(S^1,f_{1,\infty})$ has the specification property.
Consequently, the NDS $(S^1,f_{1,\infty})$ has positive topological entropy and all points are entropy point. In particular, it is topologically chaotic.
\end{example}
As an application of Corollary \ref{corollary3} we give the following example.
\begin{example}\label{example6}
Let $I=[0,1]$ and $g(x)=|1-|3x-1||$. Take a sequence of points $0=a_{1}<b_{1}=a_{2}<b_{2}=\cdots=a_{n}<b_{n}=\cdots$ converging to $1$. Define $\phi\in\mathcal{C}(I,I)$ such that
$\phi|_{[a_{n},b_{n}]}=\sigma_{n}^{-1}\circ g^{n}\circ \sigma_{n}$ for each $n\geq 1$, where $\sigma_{n}$ is the unique increasing affine map from $[a_{n},b_{n}]$ onto $[0,1]$. Let
\begin{equation*}
f_{n}(x)=
\begin{cases}
  \phi(x) & \mbox{if}\ x\in[a_{n},b_{n}],\\
  x & \mbox{if}\ x\in I\setminus[a_{n},b_{n}].
\end{cases}
\end{equation*}
Then $(I,f_{1,\infty})$ is an NDS of surjective maps and $h_{\text{top}}(I, f_{n,\infty})=0$ for every $n\geq1$, see \cite[Figure 6a and comments]{KS}. Consequently, $h^{*}(I, f_{\infty})=0$. Hence, by Corollary \ref{corollary3}, the NDS $(I, f_{1,\infty})$ does not have the specification property.
\end{example}
\section*{Acknowledgements}
The authors would like to thank the respectful referee for his/her comments on the manuscript.

\end{document}